\documentclass[default]{sn-jnl}

\usepackage[algo2e, ruled]{algorithm2e}
\usepackage{amsfonts, amsmath, mathtools, amssymb}
\usepackage{subcaption}
\newcommand{\argmin}{\operatornamewithlimits{argmin}}
\newcommand{\ie}{\textit{i}.\textit{e}., }
\newcommand{\eg}{\textit{e}.\textit{g}. }
\newcommand{\etal}{\textit{et al}. }
\DeclarePairedDelimiter{\norm}{\lVert}{\rVert}
\DeclarePairedDelimiter{\ceil}{\lceil}{\rceil}

\newcommand{\fold}{f_{old}}
\newcommand{\fnew}{f_{new}}
\newcommand{\gammamin}{\gamma_\text{min}}
\newcommand{\gammamax}{\gamma_\text{max}}
\newcommand{\tmin}{t_\text{min}}
\newcommand{\tmax}{t_\text{max}}
\newcommand{\lambdamax}{\lambda_\text{max}}
\newcommand{\tbb}{t_\text{BB}}
\DeclarePairedDelimiter{\absval}{\lvert}{\rvert}
\newcommand{\E}{\mathbb{E}}
\renewcommand{\R}{\mathbb{R}}

\jyear{2022}%

\theoremstyle{thmstyleone}%
\newtheorem{theorem}{Theorem}
\newtheorem{lemma}{Lemma}
\theoremstyle{thmstyletwo}%
\newtheorem{corollary}{Corollary}%

\theoremstyle{thmstylethree}%
\newtheorem{definition}{Definition}%
\newtheorem{invariant}{Invariant}%
\begin{document}

\title[Approximately Exact Line Search]{Approximately Exact Line Search}


\author*[1]{\fnm{Sara} \sur{Fridovich-Keil}}\email{sfk@berkeley.edu}

\author[1]{\fnm{Benjamin} \sur{Recht}}\email{brecht@berkeley.edu}

\affil[1]{\orgdiv{Electrical Engineering and Computer Sciences}, \orgname{UC Berkeley}}

\abstract{We propose \emph{approximately exact line search} (AELS), which uses only function evaluations to select a step size within a constant fraction of the exact line search minimizer of a unimodal objective. We bound the number of iterations and function evaluations of AELS, showing linear convergence on smooth, strongly convex objectives with no dependence on the initial step size for three descent methods: steepest descent using the true gradient, approximate steepest descent using a gradient approximation, and random direction search. We demonstrate experimental speedup compared to Armijo line searches and other baselines on weakly regularized logistic regression for both gradient descent and minibatch stochastic gradient descent and on a benchmark set of derivative-free optimization objectives using quasi-Newton search directions. We also analyze a simple line search for the strong Wolfe conditions, finding upper bounds on iteration and function evaluation complexity similar to AELS. Experimentally we find AELS is much faster on deterministic and stochastic minibatch logistic regression, whereas Wolfe line search is slightly faster on the DFO benchmark. Our findings suggest that line searches and AELS in particular may be more useful in the stochastic optimization regime than commonly believed.
}




\maketitle

\section{Introduction}
\label{sec:intro}

Consider the unconstrained optimization problem:
\begin{align}
\argmin_{x \in \R^n} f(x)
\end{align}
where $f$ may be known analytically, derived from a large dataset, and/or available only as a black box. In our analysis, we assume that $f$ is strongly convex and has Lipschitz gradient, but our experiments include more general objectives. Algorithm~\ref{alg:gradient descent} presents a prototypical approach in our quest to minimize the objective $f$. 

\noindent
\begin{algorithm2e}[h]
\DontPrintSemicolon
\KwIn{$f, x_0, K, T_0 > 0, \beta \in (0,1)$}
$t_{-1} \gets T_0$\;
\For{$k = 0$, $k < K$, $k = k + 1$}{
$d_k \gets \text{SearchDirection}(f, x_k)$\;
$t_k \gets \text{LineSearch}(f, x_k, \beta^{-1}t_{k-1}, \beta, d_k)$\;
$x_{k+1} \gets x_k + t_k d_k$\;
}
\Return{$x_K$}
\caption{Descent Method}
\label{alg:gradient descent}
\end{algorithm2e}

In each step $k$ with iterate $x_k$, we choose a search direction $d_k$ (by convention a \emph{descent} direction)
and a step size $t_k > 0$, and update $x_k$ accordingly. Our contribution is a novel method for the LineSearch subroutine, \emph{approximately exact line search} (Algorithm~\ref{alg:approximately exact line search with warm start}, or AELS), an adaptive variant of golden section search \cite{avriel1968golden} that operates without a predefined search interval or required accuracy. As its name suggests, the core of our approach is to approximate the solution of an \emph{exact line search}, the step size that minimizes the one-dimensional slice of the objective along the search direction. Exact line search is a natural and effective heuristic for choosing the step size, but computing the minimization exactly can be too costly for a single step of Algorithm~\ref{alg:gradient descent}. AELS instead approximates this minimizer within a constant fraction, using only a logarithmic number of objective evaluations. Although many other approximate line searches exist, these generally either require additional gradient or directional derivative computations, which may be expensive or unavailable in some settings, or are inefficient, using more function evaluations or taking smaller steps than necessary.

Approximately exact line search is parameter-free, with nominal input parameters of adjustment factor $\beta$ and initial step size $T_0$ that have minimal or no influence on the iteration complexity, respectively. Although its performance bounds depend on the Lipschitz and strong convexity parameters of the objective function (and the quality of the gradient approximation as the descent direction), these parameters play no role in the algorithm itself. 
AELS requires only unimodality (not convexity) to find a good step size, and can adjust the step size as much as necessary in each step. 
It selects a step size using only function evaluations, so irregularities in the gradient have no direct impact on its execution and it applies naturally to derivative-free optimization (DFO). AELS is competitive with gradient-based line searches when accurate gradients are available, yet avoids the fragility of gradient-based line searches in settings where the gradient is noisy. In particular, we demonstrate the benefits of AELS in the more difficult line search cases of derivative-free optimization (where the gradient is often estimated by finite differencing) and minibatch stochastic gradient descent (where the gradient and function value are approximated using a subset of a large dataset at each step); in both of these cases the error in the gradient is typically larger than any error in the function evaluation.

We prove linear iteration complexity (of a flavor similar to \cite{nesterov2017random}) and logarithmic function evaluation complexity per iteration in three settings: available gradient, approximate gradient via finite differencing, and random direction search (see \cite{Stich_2013, stich2014convex}). 
AELS shows strong empirical performance on benchmark strongly convex tasks (regularized logistic regression) with full batch gradient descent (GD) and minibatch stochastic gradient descent (SGD) \cite{robbins1951stochastic}, converging faster than Armijo and Wolfe line searches and simple prespecified step size schedules. These performance gains persist over a wide range of minibatch sizes, including single-example minibatches, and a wide range of required accuracies, with performance degrading to the level of an adaptive Armijo line search as allowable relative error drops to $10^{-6}$.
AELS also shows competitive empirical performance using BFGS \cite{bfgs} search directions on benchmark nonconvex DFO objectives \cite{more2009benchmarking}, tending to converge faster than Nelder-Mead \cite{nelder1965simplex} (designed specifically for DFO) and BFGS with Armijo line searches, and similar to BFGS with a Wolfe line search \cite{wolfe1969convergence} (requires directional derivative estimation in the line search, but guarantees a step size for which the BFGS direction update is well-defined). In light of this empirical similarity, we also analyze the iteration and function evaluation complexity of Wolfe line search, and find it remarkably similar to AELS (assuming directional derivatives are approximated by two-point finite differencing).

\subsection{Outline}
We begin with a review of related work in Section~\ref{sec:related_work}. We then introduce approximately exact line search in Section~\ref{sec:algs}. Section~\ref{sec:preliminaries} reviews notation and useful preliminaries. Section~\ref{sec:theory} contains a summary of convergence rates for general Armijo line searches (Subsection~\ref{sec:general_armijo_results}), our bounds on the iteration and function evaluation complexity of approximately exact line search (Subsection~\ref{sec:approx_exact_results}), and a comparison to the iteration and function evaluation complexity of a strong Wolfe line search. These results are proven in Section~\ref{sec:proofs} and validated experimentally in Section~\ref{sec:experiments}. Our conclusions are presented in Section~\ref{sec:conclusions}.

\section{Related work}
\label{sec:related_work}

Line search is a well-studied problem, and many approaches have been proposed to address the basic goal of choosing a step size as well as various practical challenges like stochasticity and nonconvexity. In this section, we offer a survey of relevant literature, with a particular emphasis on work that inspired ours. 

\paragraph{Step size schedules}
The simplest step size is a constant, which will suffice for convergence provided that the step size is no more than $\frac{1}{L}$ for an objective with $L$-Lipschitz gradient \cite{recht-wright}. In practice, however, even for $L$-smooth objectives $L$ is rarely known, and an improperly chosen fixed step size can result in divergence or arbitrarily slow convergence. Another simple strategy is a decreasing step size, for instance $t_k \propto \frac{1}{k}$. This has the benefit of guaranteeing convergence, but is still susceptible to slow convergence if the initial step size is too small  \cite{Bengio_2012}. 
Other schedules include the Polyak step size \cite{Polyak87} and a hierarchical variant, which adaptively updates a lower bound on the objective to choose the step sizes without line search \cite{hazan2019revisiting}. The Barzilai-Borwein method \cite{bb} uses history of the prior iterate and gradient to adapt to local curvature and compute a step size so that the update approximates a Newton step, without using the Hessian.

\paragraph{Stochastic methods}
Along with these strategies, many heuristics and hand-tuned step size schedules exist, especially in the stochastic setting. Engineers may hand-tune a series of step sizes during the course of optimization \cite{Bengio_2012}. These strategies can be successful, but require substantial human effort which can be averted by a line search algorithm. Several variations of line search have been proposed (and proven) specifically for the stochastic setting; these methods typically adjust the step size by at most a constant factor in each step. Paquette and Scheinberg \cite{paquette2018stochastic} prove convergence of such a method, with careful tracking of expected function decrease in each step. Likewise, Berahas, Cao, and Scheinberg prove convergence of a slightly simpler line search \cite{berahas2019global} in the noisy setting, provided sufficient bounds on the noise. Another class of methods diverges from the classical framework of Algorithm~\ref{alg:gradient descent} by borrowing from online learning \cite{orabona2016coin, cutkosky2017online, Duchi:2011:ASM:1953048.2021068} to choose each iterate without line search, also with proven convergence in the stochastic setting \cite{orabona2017training}. Many methods for the stochastic setting also include variance reduction, such as adaptively choosing the minibatch size in each iteration \cite{Bollapragada_2018, bollapragada2018progressive, Friedlander_2012} or accumulating information from prior gradients \cite{schmidt2017minimizing, schaul2012pesky, almeida_langlois_amaral_plakhov_1999}. In the case of interpolating classifiers, Vaswani \etal \cite{vaswani2019painless} prove that backtracking line search with a periodic step size reset converges at the same rate in stochastic (minibatch) and deterministic (full batch) gradient settings. 

Line search methods have also been studied in the DFO setting; see \cite{DFObook} for an overview. Jamieson, Nowak, and Recht \cite{jamieson2012query} offer lower bounds on the number of (potentially noisy) function evaluations, Nesterov and Spokoiny \cite{nesterov2017random} and Golovin \etal \cite{golovin2020gradientless} analyze methods with random search directions, and Berahas \etal \cite{berahas2019theoretical} compare various gradient approximation methods. 

\paragraph{Armijo condition methods}
Perhaps the most common line search is Armijo backtracking, which begins with a fixed initial step size and geometrically decreases it until the Armijo condition is satisfied \cite{recht-wright, nocedal2006line, nesterov2013gradient}. 

\begin{definition}
\label{armijo}
Armijo condition \cite{armijo1966minimization}: 
If our line search direction is $d$ (such that $d^T\nabla f(x) < 0$) from iterate $x$, then the Armijo condition is
\begin{align}
f(x + t d) \leq f(x) + c_1 t d^T\nabla f(x)\;,
\end{align}
where $c_1 \in (0,1)$. In our analysis we use $c_1 = \frac{1}{2}$ for convenience; in experiments we use $c_1 = 10^{-4}$ (as recommended in \cite{nocedal2006line}).
In the case of steepest descent, we have $d = -\nabla f(x)$, with corresponding Armijo condition
\begin{align}
f(x - t \nabla f(x)) \leq f(x) - c_1 t \norm{\nabla f(x)}^2\;.
\end{align}
\end{definition}

On convex objective functions with $L$-Lipschitz gradients, such a strategy (taking Armijo steps in the negative gradient direction) guarantees convergence with a rate depending on the chosen initial step size, which is intended to be $\geq \frac{1}{L}$ \cite{recht-wright, nesterov2013gradient}. As we might expect, this implementation (denoted ``traditional backtracking'' in our experiments) performs well with a sufficiently large initial step size but can converge quite slowly if the initial step size is too small. More sophisticated Armijo line searches attempt to find a reasonably large step size that satisfies Definition~\ref{armijo}. One option is to replace the fixed initial step size with a value slightly larger than the previous step size \cite{nesterov2013gradient} (denoted ``adaptive backtracking'' in our experiments), which like many of the stochastic line searches increases the step size by at most a constant factor in each step. Another is the polynomial interpolation method of Nocedal and Wright, which seeks to minimize a local approximation of the objective while satisfying Definition~\ref{armijo} \cite{nocedal2006line}. Although the Armijo condition requires no additional gradient evaluations, it does rely on the gradient at the iterate, making it more sensitive (than approximately exact line search, which uses no gradient information) to noisy gradients like those available in stochastic or derivative-free optimization.

\paragraph{Curvature condition methods}
In addition to the Armijo condition, which ensures that the step size is not too large, many line searches include curvature conditions, such as the Wolfe or Goldstein conditions \cite{wolfe1969convergence, nocedal2006line}, which aim to ensure that the step size is also not too small. Although this largely remedies the small step size failure mode of pure Armijo backtracking methods, it comes at a cost. The Wolfe curvature condition requires gradient or directional derivative measurements during the line search, which may be expensive, unavailable, or noisy in certain situations (\eg stochastic optimization or DFO). The Goldstein condition closely parallels the Armijo condition and requires no extra gradient evaluations, but depending on the objective it may exclude the exact line search minimizer \cite{nocedal2006line}. 

\paragraph{Bracketing methods}
If the exact line search step size is known to lie within a certain interval, intermediate function evaluations can reduce the search interval geometrically until a desired accuracy is achieved \cite{brent1971algorithm, avriel1968golden}, and this search interval can be found adaptively (see \eg Algorithms 3.5 and 3.6 in \cite{nocedal2006line}, although these use gradient evaluations). Our proposed approximately exact line search (Algorithm~\ref{alg:approximately exact line search with warm start}) is particularly inspired by golden section search \cite{avriel1968golden}, which minimizes any unimodal 1D function in logarithmic time (as a function of the initial search interval and the desired precision). Golden section search proceeds in a manner much like binary search, reducing the size of the active search window by a constant fraction with each additional function evaluation, and uses no gradient evaluations. AELS (Algorithm~\ref{alg:approximately exact line search with warm start}) relies on the same logic to bracket the minimum with only function evaluations, but operates without a known search interval or fixed precision. Instead, we use a geometrically-increasing sequence of trial step sizes to determine the search interval, and search to a relative precision that guarantees a step size within a constant fraction of the true minimizing step size. This relative precision allows AELS to smoothly trade off the function complexity of the line search and the precision of the selected step size: when we are close to the minimum, our line search has high precision, whereas farther away we accept lower precision in favor of proceeding to the next step.

\section{Proposed algorithm}
\label{sec:algs}

The goal of approximately exact line search (Algorithm~\ref{alg:approximately exact line search with warm start}), as its name suggests, is to find a step size within a constant fraction of (and without exceeding) the exact line search minimizer. 
AELS achieves this goal by iteratively increasing and/or decreasing the step size by a fixed multiplicative factor as long as doing so decreases the line search objective. 

Our procedure follows the same logic as golden section search \cite{avriel1968golden}: a set of three sample step sizes contains the exact line search minimizer of a unimodal 1D objective as long as, of our three samples, the one in the middle has the lowest objective value. Accordingly, we evaluate the objective at an exponentially spaced set of step sizes, stopping when we know the exact minimizer lies in the interior of our sample step sizes. Among the three sample step sizes that most closely bracket the minimizer, we use the smallest to ensure that we never exceed the true minimizer. A variant of our algorithm instead chooses the sampled step size with the lowest objective value; this variation performs indistinguishably from AELS both in theory and in our experiments, so we exclude it for clarity.

Using only function evaluations to choose the step size, rather than relying on the Armijo condition, requires a small overhead in function evaluations per step but enables greater flexibility in the objective functions we can optimize efficiently (by removing any gradient requirements). We also note that the ``warm start'' procedure of initializing the step size slightly larger than the previous step size used ($T = \beta^{-1}t$) is beneficial experimentally, although an alternative approach of initializing with a step size equal to the previous step size ($T = t$) achieves the same worst-case bounds in terms of steps and function evaluations. Since AELS chooses a step size slightly smaller than an exact line search (rather than slightly larger), this larger initialization may on average better approximate the desired exact line search minimizer in practice.

A {\tt python} implementation of AELS, including various search directions, special cases to optimize nonconvex objectives (like those in our DFO experiments), and adaptations for finite precision function evaluations (which may be indistinguishable in a wide, flat minimum), is provided at \url{https://github.com/modestyachts/AELS}. 

\noindent
\begin{algorithm2e}[t]
\DontPrintSemicolon
\SetNoFillComment
\KwIn{$f, x, T > 0, \beta \in (0,1), d$}
$t \gets T$\;
$\fold \gets f(x)$\;
$\fnew \gets f(x + t d)$\;
\tcc{Decide to try increasing or decreasing the step size}
$\alpha \gets \beta$\;

\If{$\fnew \leq \fold$}{
$\alpha \gets \beta^{-1}$\;
}
\tcc{Increase/Decrease the step size as long as $f$ is decreasing}
\Repeat{$\fnew \geq \fold$}{ 
$t \gets \alpha t$\;
$\fold \gets \fnew$\;
$\fnew \gets f(x + td)$\;
} 
\tcc{If increasing the step size failed, try decreasing it}
\If{$t = T/\beta$}{
$t \gets T$\;
$\alpha \gets \beta$\;
$(\fnew, \fold) \gets (\fold, \fnew)$\;
\Repeat{$\fnew > \fold$}{
$t \gets \alpha t$\;
$\fold \gets \fnew$\;
$\fnew \gets f(x + td)$\;
}
}
\tcc{Return a step size $\leq$ the exact line search minimizer}
\If{$\alpha < 1$}{
\Return{$t$}\;
}
\Return{$\beta^2 t$}\;
\caption{Approximately Exact Line Search}
\label{alg:approximately exact line search with warm start}
\end{algorithm2e}

\begin{figure}[h]
\centering
\begin{subfigure}[h]{0.325\textwidth}
        \includegraphics[width=\textwidth]{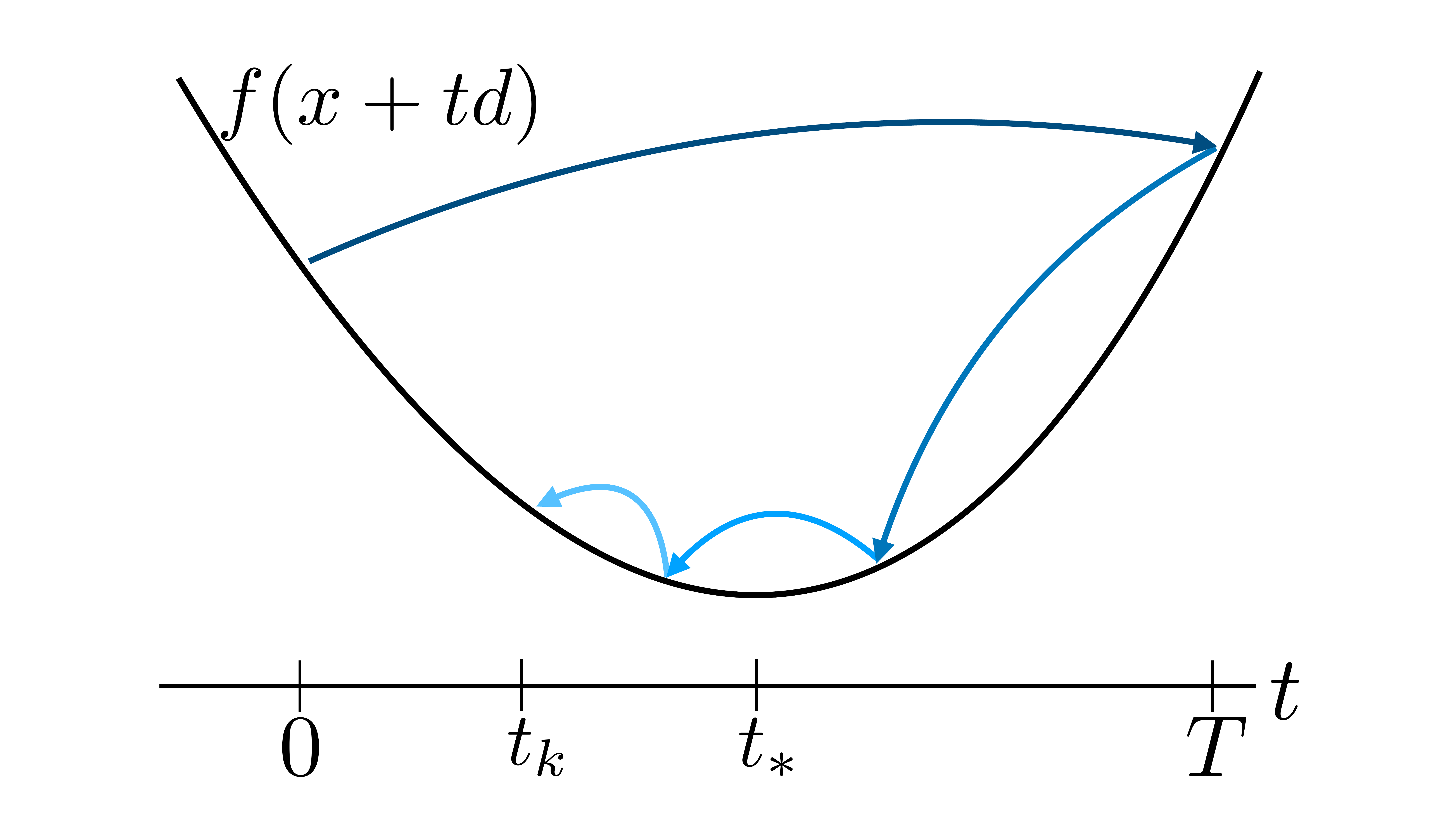}
        \caption{$T$ too large}
        \label{fig:big}
    \end{subfigure}
    \begin{subfigure}[h]{0.325\textwidth}
        \includegraphics[width=\textwidth]{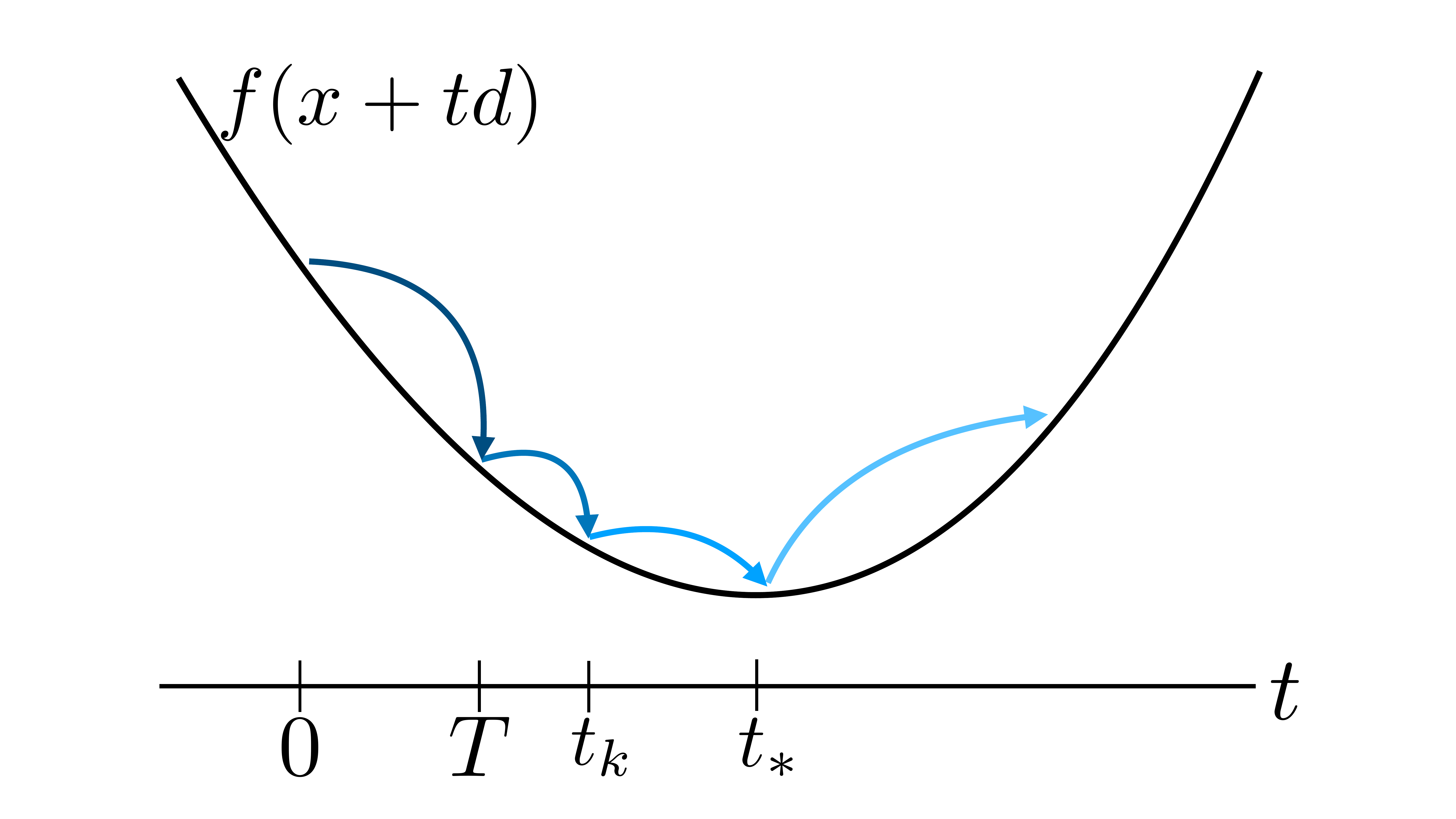}
        \caption{$T$ too small}
        \label{fig:small}
    \end{subfigure}
    \begin{subfigure}[h]{0.325\textwidth}
        \includegraphics[width=\textwidth]{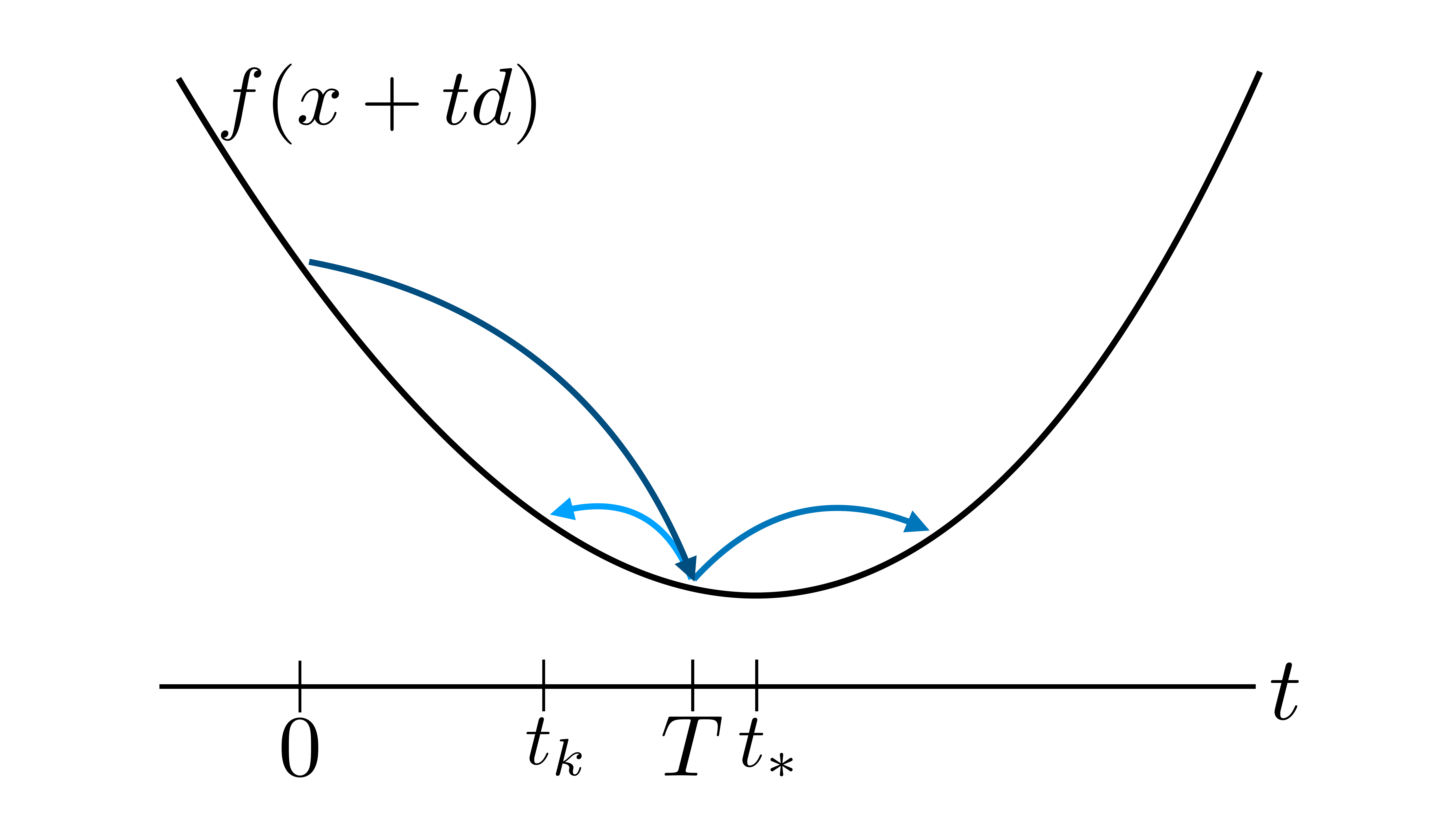}
        \caption{$T$ near $t_*$}
        \label{fig:medium}
    \end{subfigure}
\caption{Illustration of Algorithm~\ref{alg:approximately exact line search with warm start} in three characteristic scenarios defined by the relationship between $T$ and $t_*$, with the chosen step size denoted $t_k$. Arrows indicate trials of candidate step sizes, with each candidate step shown in a lighter shade than the previous. In this illustration and in our experiments, we use $\beta = 2/(1+\sqrt{5})$, the inverse golden ratio.}
\label{fig:alg_execution_illustration}
\end{figure}

Figure~\ref{fig:alg_execution_illustration} illustrates the three regimes of AELS. If $T$ is sufficiently larger than the exact line search minimizer $t_*$ (Figure~\ref{fig:big}), our first candidate step will immediately increase the objective value, so we iteratively decrease the step size until we are guaranteed that our step size $t_k$ does not exceed $t_*$. If $T$ is sufficiently smaller than $t_*$ (Figure~\ref{fig:small}), we iteratively increase the candidate step size until it surpasses $t_*$, and then select the largest previous candidate step size $t_k$ that we can guarantee does not exceed $t_*$. If $T$ is sufficiently close to $t_*$ (Figure~\ref{fig:medium}), we must try both larger and smaller step sizes to find a $t_k$ we can guarantee is within a constant fraction of $t_*$.

\section{Preliminaries}
\label{sec:preliminaries}

We begin by compiling useful facts and notation that appear in our analysis.

\subsection{Notation}

We use the following notation:
\begin{itemize}
    \item $\norm{\cdot} \equiv \norm{\cdot}_2$, the Euclidean norm.
    \item $L$ denotes the Lipschitz constant of the gradient of $f$.
    \item $m$ denotes the strong convexity parameter of $f$.
    \item $f'_v(x) \equiv \frac{v^T\nabla f(x)}{\norm{v}}$, the directional derivative of $f$ at $x$ in direction $v$.
    \item $f_*$ denotes the minimum value of a (typically convex) function $f$. $x_*$ denotes an optimal input point, such that $f(x_*) = f_*$.
    \item $\Delta$ is the user-specified tolerance for $\norm{x-x_*}$ at termination. We assume that $\norm{x_k-x_*} \geq \Delta$ during algorithm execution.
\end{itemize}

\begin{definition}
\label{Lipschitz gradient}
A function $f$ has $L$-Lipschitz gradient if for all $x, y$:
\begin{align}
\norm{\nabla f(x) - \nabla f(y)} \leq L \norm{x - y}\;.
\end{align}
Equivalently (using Taylor's theorem), $f$ has $L$-Lipschitz gradient if for all $x, y$:
\begin{align}
f(y) \leq f(x) + \nabla f(x)^T(y-x) + \frac{L}{2}\norm{y-x}^2\;.
\end{align}
\end{definition}

\begin{definition}
\label{strong convexity}
A function $f$ is $m$-strongly convex if for all $x, y$:
\begin{align}
\norm{\nabla f(x) - \nabla f(y)} \geq m \norm{x - y}\;.
\end{align}
Equivalently (using Taylor's theorem), $f$ is $m$-strongly convex if for all $x, y$:
\begin{align}
f(y) \geq f(x) + \nabla f(x)^T(y-x) + \frac{m}{2}\norm{y-x}^2\;.
\end{align}
\end{definition}

\subsection{Gradient approximation}
For derivative-free optimization, we approximate the gradient or directional derivative using finite differences. In this setting, we can often guarantee performance that is nearly as fast as in the gradient setting, but with an extra error term determined by the quality of our gradient approximation. 

\begin{lemma}
\label{directional derivative approximation}
Let $v$ be any unit vector and $\sigma > 0$ be a sampling radius. Let $g_v^f(x) = \frac{f(x+\sigma v) - f(x)}{\sigma}$ be a forward finite differencing approximation of $f'_v(x)$ with radius $\sigma$. Then for all $x$:
\begin{align}
\absval{g_v^f(x) - f'_v(x)} \leq \frac{L\sigma}{2}\;.
\end{align}
\begin{proof}
See Theorem 2.1 in \cite{berahas2019theoretical}.
\end{proof}
\end{lemma}

\begin{corollary}
\label{directional derivative corollary}
A forward finite differencing gradient approximation $g(x)$ with sampling radius $\sigma$ satisfies $\norm{g(x) - \nabla f(x)} \leq \frac{L\sigma \sqrt{n}}{2}$, where $n$ is the dimension of $x$.
\end{corollary}

We work with forward finite differencing rather than central finite differencing because although central finite differencing can give better accuracy, it requires twice as many function evaluations.

\subsection{Function properties}
Although we show experimentally that AELS is competitive even on nonconvex objective functions with potentially non-Lipschitz gradients, we prove convergence rates in the simpler setting of strongly convex functions with Lipschitz gradients. Functions with various subsets of these properties satisfy relationships that will be useful in our analysis.

\begin{lemma}
\label{strongly convex property}
If $f$ is $m$-strongly convex, then for all $x$:
\begin{align}
f(x) - f_* \leq \frac{1}{2m}\norm{\nabla f(x)}^2\;.
\end{align}
\begin{proof} See chapter 3 in \cite{recht-wright}.
\end{proof}
\end{lemma}

\begin{lemma}
\label{taylor and lipschitz and strong convexity}
Let $f$ be any $m$-strongly convex function with $L$-Lipschitz gradient, $x$ be the current iterate, $d$ be any vector with $d^T \nabla f(x) < 0$, and $t > 0$ be any step size. Then for all $x$:
\begin{align}
f(x) + t d^T \nabla f(x) + \frac{mt^2}{2}\norm{d}^2 \leq f(x + td) \leq f(x) + t d^T \nabla f(x) + \frac{Lt^2}{2}\norm{d}^2\;.
\end{align}
Note that if we set $d = -f'_v(x)v$ for an arbitrary unit vector $v$, then $d^T\nabla f(x) = -(f'_v(x))^2 < 0$ and
\begin{align}
f(x) - t\left(1-\frac{tm}{2}\right)(f'_v(x))^2 \leq f(x - tf'_v(x)v) \leq f(x) - t\left(1-\frac{tL}{2}\right)(f'_v(x))^2\;.
\end{align}
Regardless of the choice of $d$, the first inequality requires only $m$-strong convexity and the second inequality requires only $L$-Lipschitz gradient.
\begin{proof}
See chapter 2 in \cite{Nesterovbook}.
\end{proof}
\end{lemma}

\section{Theoretical results}
\label{sec:theory}

In this section, we state and discuss convergence rates for general Armijo line searches, AELS, and a strong Wolfe line search; proofs are deferred to Section~\ref{sec:proofs}. We focus on objective functions that are strongly convex and have Lipschitz gradients.

We analyze three schemes for selecting $d_k$ (the search direction in Algorithm~\ref{alg:gradient descent}):
\begin{enumerate}
\item $d_k = -\nabla f(x_k)$.
\item $d_k$ is an approximation of $-\nabla f(x_k)$ with bounded error, \eg by using $\sigma$-finite differencing along coordinate directions. We assume $d_k$ is a descent direction.
\item $d_k = -\mu_k v_k$, where $v_k$ is a uniformly random unit vector and $\mu_k$ is a $\sigma$-finite differencing approximation of the directional derivative in direction $v_k$. We assume $\mu_k$ has the correct sign, so $d_k$ is a descent direction.
\end{enumerate}

The first scheme is useful for objective functions which are efficiently differentiable, while the second and third are useful for derivative-free optimization. 

We begin by analyzing general line searches whose step sizes satisfy the Armijo condition and then proceed with our analysis of AELS, followed by Wolfe line search. All three types of line search guarantee linear convergence (modulo gradient approximation errors), but AELS has two advantages in its convergence bounds compared to Armijo methods. 

The first is that AELS' convergence rates (Theorem~\ref{approx exact strongly convex iteration rate}) depend only on the objective function, with no dependence on the initial step size $T_0$. For Armijo methods, the convergence rates (Theorem~\ref{general armijo strongly convex rate}) depend on the minimum step size taken, which can be needlessly slow if $T_0 \ll \frac{1}{L}$ and the line search is insufficiently adaptive. 

The second advantage appears in random direction search, in which the expected convergence rate is truly linear for AELS, whereas an Armijo line search has an additive error term from the gradient approximation used in the Armijo condition.  

One downside of AELS is that it requires a slight overhead in function evaluations (see Theorem~\ref{approx exact strongly convex function rate}) compared to most Armijo line searches, to compensate for not using gradient information. In practice, this overhead is minimal compared to the speedup of convergence in fewer iterations, as shown in Section~\ref{sec:experiments}.

\subsection{Convergence rates for general Armijo line searches}
\label{sec:general_armijo_results}

\begin{definition}
\label{def:t_a}
$t_a$ is the positive solution to $f(x + t d) = f(x) + \frac{t}{2} d^T\nabla f(x)$, the Armijo condition (Definition~\ref{armijo}) in direction $d$ (with $d^T\nabla f(x) < 0$) from iterate $x$. 

Note that in some cases we work with an approximation to the gradient, in which case we replace the gradient $\nabla f(x)$ in this condition with its approximation $g(x)$.
\end{definition}

We begin with upper bounds on the iteration complexity of general line searches based on the Armijo condition (Definition~\ref{armijo}) using our three gradient approximation methods (true gradient, finite differencing, and random direction). These bounds apply to any line search whose step sizes satisfy the Armijo condition, such as the traditional backtracking line search and its adaptive version we consider in Section~\ref{sec:experiments}. As is evident in the bounds, the convergence rate of Armijo-based methods depends on the minimum step size they take. Traditional backtracking line search, for instance, never increases the step size from a user-provided initial guess, which may be arbitrarily small compared to $t_a$, the largest step size that satisfies the Armijo condition (Definition~\ref{def:t_a}). Adaptive backtracking line search can increase the step size, but may require many steps to reach a step size near $t_a$.

\begin{lemma}
\label{general armijo strongly convex rate}
Let $f$ be $m$-strongly convex and have $L$-Lipschitz gradient, and let $x_{k+1} = x_k - t_k g(x_k)$, where $t_k \in [\tmin, \tmax]$ satisfies the Armijo condition with respect to descent direction $-g(x_k)$ (\ie $g(x_k)$ is the gradient approximation and $-g(x_k)$ is the search direction) and $n$ is the dimension of $x$.

Using the true gradient,
\begin{align}
f(x_k) - f_* \leq (1-m\tmin)^k(f(x_0)-f_*)\;.
\end{align}

Using a $\sigma$-finite differencing gradient approximation (assuming that the approximate negative gradient is a descent direction),
\begin{align}
f(x_{k+1}) - f_* \leq \min \Big\{& (1-m\tmin)^k(f(x_0) - f_*) + \frac{L^2 \sigma \sqrt{n} \tmax \norm{x_0 - x_*}}{2m\tmin},  \nonumber \\&(1-\frac{m\tmin}{2})^k(f(x_0) - f_*) + \frac{L^2 \sigma^2 n \tmax}{4m\tmin} \Big\}\;. 
\end{align}

Using a random direction and approximating the directional derivative by $\sigma$-finite differencing (such that the search direction is a descent direction),
\begin{align}
& \hspace{-2.8cm}\E[f(x_k) - f_*] \leq \nonumber \\
\min\Big\{(f(x_0)-f_*)& \Big(1- \frac{m \tmin}{n} \Big)^k + \frac{ L^2\sigma n \tmax\norm{x_0 - x_*}}{2m \tmin}, \nonumber \\
 (f(x_0)-f_*) &\Big(1- \frac{m \tmin}{2n} \Big)^k + \frac{L^2\sigma^2n\tmax}{4m\tmin}  \Big\} \;.
\end{align}
\end{lemma}

Lemma~\ref{general armijo strongly convex rate} shows in general linear convergence with a worst-case rate determined by the strong convexity parameter $m$ of the objective and the minimum step size taken during the line search: the greater the curvature of the objective and the larger the steps, the faster the convergence. In particular, using Lemma~\ref{t_a bounds} we can bound the minimum step size for any Armijo line search. For traditional backtracking line search with initial step size $T_0$ using the true gradient, $\tmin \geq \min(T_0, \frac{\beta}{L})$. For adaptive backtracking (which increases the initial step size by a constant factor $\beta^{-1}$ in each step) using the true gradient, $\tmin \geq \min(\frac{T_0}{\beta^k}, \frac{\beta}{L})$ in step $k$. We also experimented with a ``forward-tracking'' Armijo line search, which increases the initial step size as much as necessary in each step to guarantee $\tmin \geq \frac{\beta}{L}$, assuming the true gradient is used. This method guarantees a step size within a $\beta$ fraction of the maximum step size that satisfies the Armijo condition; we found it to be more efficient than both traditional and adaptive backtracking line searches, but slower than AELS, which bypasses the Armijo condition entirely. 

Approximating the gradient by $\sigma$-finite differencing slows the convergence rate by an additive error term, which goes to zero as $\sigma \to 0$. This error term is larger for higher-dimensional objectives (since each dimension of the gradient must be estimated), ill-conditioned objectives (with large $L$ and small $m$), and objective-algorithm pairs with large variation in step size (often related to poor conditioning). The two-term rate means we expect linear convergence until a certain accuracy threshold, beyond which we can no longer improve because the magnitude of the true gradient is overwhelmed by approximation error. The accuracy threshold can be dropped by decreasing the sampling radius $\sigma$ (assuming noiseless function evaluations); in practice this could be done adaptively to achieve a desired precision with performance mostly adhering to the linear term until progress becomes limited by numerical precision. 

If descent proceeds along random directions with approximate directional derivatives, the linear rate is worsened by a factor of the dimension (as observed in \cite{nesterov2017random}), with a similar $\sigma$-dependent accuracy threshold that can be managed by (potentially adaptive) reduction of $\sigma$.

\subsection{Approximately exact line search}
\label{sec:approx_exact_results}

\begin{definition}
\label{def:t_*}
$t_* \equiv \argmin_t f(x + td)$, the exact minimizer of a line search along direction $d$ (with $d^T\nabla f(x) < 0$) from iterate $x$.
\end{definition}

The core idea of AELS is to find a step size close to (and no greater than) $t_*$. More precisely, we have the following invariant in each step of Algorithm~\ref{alg:approximately exact line search with warm start}.

\begin{invariant}
\label{approx exact invariant}
Let $t$ be the step size used in a step of Algorithm~\ref{alg:approximately exact line search with warm start}. Then $t \in [\beta^2 t_*, t_*]$, where $t_*$ is the step size used in exact line search (using the same search direction from the same iterate).
\end{invariant}

Although AELS may choose a step size that does not satisfy the Armijo condition, it does so only if the larger step size achieves a smaller objective value compared to $t_a$. This property is central to the proof of Theorem~\ref{approx exact strongly convex iteration rate}, which bounds the iteration complexity of AELS.

\begin{theorem}
\label{approx exact strongly convex iteration rate}
On $m$-strongly convex objectives with $L$-Lipschitz gradients, Algorithm~\ref{alg:approximately exact line search with warm start} satisfies the following convergence rates (assuming each step follows a descent direction):

Using the true gradient,
\begin{align}
f(x_k) - f_* \leq \left(1-\beta^2\left(1-\sqrt{1-\frac{m}{L}} \right) \right)^k (f(x_0)-f_*) \;.
\end{align}

Using a $\sigma$-finite differencing gradient approximation with $\sigma = \frac{2\tilde \sigma m \Delta}{L\sqrt n}$ and $\tilde \sigma < 1$ (so that the negative gradient approximation is a descent direction),
\begin{align}
& \hspace{-0.85cm} f(x_{k+1}) - f_* \leq  \nonumber \\
\min \Big\{& \Big(1-\beta^2\Big(1-\sqrt{1-\frac{m}{L}}\Big)\frac{1-\tilde\sigma}{(1+\tilde\sigma)^2}\Big)^k(f(x_0) - f_*) + \frac{\tilde\sigma(1+\tilde\sigma)^2\Delta L \norm{x_0 - x_*}}{(1-\tilde\sigma)^2\beta^2\Big(1-\sqrt{1-\frac{m}{L}}\Big)}\;, \nonumber \\
&\Big(1-\frac{\beta^2}{2}\Big(1-\sqrt{1-\frac{m}{L}}\Big)\frac{1-\tilde\sigma}{(1+\tilde\sigma)^2}\Big)^k(f(x_0) - f_*) + \frac{\tilde\sigma^2m(1+\tilde\sigma)^2\Delta^2}{(1-\tilde\sigma)^2\beta^2\Big(1-\sqrt{1-\frac{m}{L}}\Big)} \Big\}\;. 
\end{align}

Using a random direction and approximating the directional derivative (by any method with the correct sign),
\begin{align}
\E[f(x_k) - f_*] \leq \left(1- \frac{\beta^2}{n}\left(1-\sqrt{1-\frac{m}{L}}\right)  \right) ^k(f(x_0)-f_*)  \;,
\end{align}
where $n$ is the dimension of $x$.
\end{theorem}

Note that Theorem~\ref{approx exact strongly convex iteration rate}, like Lemma~\ref{general armijo strongly convex rate}, assumes that the search direction in each step is a descent direction, \ie that the gradient or directional derivative estimation is sufficiently accurate. As the iterates approach the minimum, or any low-gradient region, this assumption may require adjusting the finite differencing radius $\sigma$, or other adaptations. In our implementation, we use a fixed $\sigma$ based on numerical precision (as recommended by Nocedal and Wright \cite{nocedal2006line}), and add a ``patience'' parameter to avoid endless searching: if AELS fails to improve on the function value after a specified number of line search trials (we use 20 in our experiments), we abandon the step and try again.

Theorem~\ref{approx exact strongly convex iteration rate} parallels Lemma~\ref{general armijo strongly convex rate}, with a difference in the random direction regime, in which AELS is robust to directional derivative approximation error. More generally, AELS has the feature of choosing the step size based only on function evaluations, so the performance of the line search itself has no dependence on the accuracy of the gradient approximation. Indeed in the stochastic setting (explored in Figure~\ref{fig:logistic_sgd_batchsize}) AELS performs much better with smaller minibatches (less accurate gradient and function evaluations) compared to the baseline Armijo line searches.

\begin{theorem}
\label{approx exact strongly convex function rate}
On $m$-strongly convex objectives with $L$-Lipschitz gradients, Algorithm~\ref{alg:approximately exact line search with warm start} uses at most
\begin{align}
&k\left(4 + \ceil*{\log_{1/\beta}\left(\frac{\gammamax(1+\sqrt{1-\frac{m}{L}})}{\gammamin(1-\sqrt{1-\frac{m}{L}})}\right)} \right) \nonumber \\
&~~~~~~~~ + \ceil*{\log_{1/\beta}\max\left(\frac{\beta^2\gammamin(1-\sqrt{1-\frac{m}{L}})}{mT_0}, \frac{mT_0}{\beta\gammamax(1+\sqrt{1-\frac{m}{L}})}\right)}\;.
\end{align}
function evaluations for $k$ steps, where $\gamma_k = \frac{-d_k^T\nabla f(x_k)}{\norm{d_k}^2} \in [\gammamin, \gammamax]$ according to Lemma~\ref{gamma bounds}.
\end{theorem}

The first step requires some overhead to adjust the initial step size, and all subsequent steps require up to four function evaluations, plus some potential overhead depending on the conditioning of the objective and the quality of the gradient approximation (appearing through $\gamma_k$). AELS tends to require slightly more function evaluations per step compared to simpler line searches like the backtracking Armijo methods we consider in Section~\ref{sec:experiments}
because it relies on function evaluations only, and at least three function evaluations are required just to bracket the minimizer. 

Note that Theorem~\ref{approx exact strongly convex function rate} counts only function evaluations (not gradient evaluations), so if the gradient is approximated by forward finite differencing that will cost an additional $n$ function evaluations per step. Using a random direction and approximating the directional derivative by forward finite differencing requires one additional function evaluation per step.

\subsection{Strong Wolfe line search}

We also provide upper bounds on the iteration and function evaluation complexity of Wolfe line search in the same setting: an objective function $f$ that has $L$-Lipschitz gradient and $m$-strong convexity. As described in Section~\ref{sec:related_work}, the Wolfe conditions augment the Armijo condition to guarantee both sufficient decrease in the function value and sufficient decrease in its gradient; it is this second ``curvature condition'' that guarantees a step size that is not too small. 

\begin{definition}
\label{wolfe}
Strong Wolfe curvature condition \cite{wolfe1969convergence, nocedal2006line}: 
If our line search direction is $d$ (such that $d^T\nabla f(x) < 0$) from iterate $x$, then the curvature condition is
\begin{align}
\absval{d^T\nabla f(x + t d)} \leq c_2 \absval{d^T\nabla f(x)}\;,
\end{align}
where $c_2 \in (c_1, 1)$, where $c_1$ is the parameter in the Armijo condition \ref{armijo}.
\end{definition}

Our first goal is to bound the iteration complexity of this line search, which follows directly using Lemma~\ref{general armijo strongly convex rate}, since any step size that satisfies the strong Wolfe conditions must also satisfy the Armijo condition. To apply this lemma, what remains is to bound the minimum and maximum possible step sizes that satisfy the strong Wolfe conditions. Let $t_w$ be a step size that satisfies the strong Wolfe conditions.

\begin{lemma}
\label{t_w bounds}
If $f$ is $m$-strongly convex with $L$-Lipschitz gradient and $d$ is a search (descent) direction, then
\begin{align}
t_w \in  \frac{-d^T\nabla f(x)}{\norm{d}^2} \left[\frac{1 - c_2}{L}, \frac{1}{m} \right]\;.
\end{align}
Note that the first inequality requires only the Lipschitz gradient condition, and the second requires only the strong convexity condition. Without strong convexity (as $m \to 0$) we approach the one-sided bound $t_w \geq \frac{-d^T\nabla f(x)(1-c_2)}{L\norm{d}^2}$.
\begin{proof}
The upper bound is the same as in Lemma~\ref{t_a bounds}. The lower bound is identical to that in the proof of Theorem 3.2 in \cite{nocedal2006line}, which instead uses the weak Wolfe conditions (which are always satisfied if the strong Wolfe conditions are satisfied). 
\end{proof}
\end{lemma}

\paragraph{Rate comparison}
How does this range compare to the step sizes chosen by AELS or pure Armijo line searches (like traditional and adaptive backtracking)? Using the notation $\gamma = \frac{-d^T\nabla f(x)}{\norm{d}^2}$, AELS always selects a step size larger than $\frac{\gamma \beta^2}{2L}$ (the limit of Invariant~\ref{approx exact invariant} with Lemma~\ref{t_* bounds} as $m \rightarrow 0$). In the best case (where $m = L$), AELS always selects a step size larger than $\frac{\gamma \beta^2}{L}$. In the best case for a pure Armijo line search, in which the initial step size is sufficiently large, the selected step size is larger than $\frac{\gamma \beta}{L}$. In the same notation, the step size chosen by a Wolfe line search is at least $\frac{\gamma (1-c_2)}{L}$. These lower bounds are all identical, up to small dependence on the parameters of each algorithm and the assumption of sufficient initial step size for pure Armijo methods, which is exactly the assumption that both the Wolfe curvature condition and AELS remove. Based on these bounds, we expect AELS and Wolfe line search to enjoy similar iteration complexity to each other and to a well-initialized backtracking Armijo line search, when the search direction is the negative gradient or its approximation. When the search is in a random descent direction, AELS enjoys a benefit compared to the other algorithms because its search is independent of the norm of the search direction.

\paragraph{Function evaluation complexity comparison} 
In addition to bounding the iteration complexity of a descent method with a strong Wolfe line search (with Lemma~\ref{general armijo strongly convex rate} and Lemma~\ref{t_w bounds}), we also bound the number of function evaluations required for each line search, so that we can compare the function evaluation complexity of Wolfe line search and AELS. 
Specifically, we analyze the strong Wolfe condition line search presented in Algorithms 3.5 and 3.6 of Nocedal and Wright \cite{nocedal2006line}, with the simplification that successive trial step lengths in Algorithm 3.5 are chosen by scaling by a factor $\beta^{-1}$, where $\beta \in (0, 1)$, and that interpolation in Algorithm 3.6 is done by bisection.  These slightly simplified algorithms are reproduces in Algorithm~\ref{alg:wolfe} and Algorithm~\ref{alg:zoom}. The bound is presented in Lemma~\ref{Wolfe function complexity}, proven in Section~\ref{sec:wolfe_feval}, and discussed below.

\begin{lemma}
\label{Wolfe function complexity}
The strong Wolfe line search defined by Algorithms 3.5 and 3.6 in \cite{nocedal2006line}, with the simplifications that successive trial step lengths in Algorithm 3.5 are chosen by scaling by a factor $\beta^{-1}$, where $\beta \in (0, 1)$, interpolation in Algorithm 3.6 is done by bisection, and directional derivatives are approximated by two-point finite differencing, requires no more than
$$
2 + 2\log_2\left(\frac{L}{mc_2}\right) + \max\left(2+2\log_{1/\beta}\left(\frac{-d^T\nabla f(x)}{mT\norm{d}^2}\right), \log_2\left(\frac{-TL\norm{d}^2}{d^T\nabla f(x)}\right)\right)
$$
function evaluations per line search, where $T$ is the initial step size, $d$ is the search (descent) direction, and the objective $f$ is $m$-strongly convex with $L$-Lipschitz gradient.
\end{lemma}

Assuming efficient two-point directional derivative estimation, we can compare the per-iteration function evaluation complexity of strong Wolfe and Approximately Exact line searches. Strong Wolfe requires no more than $2 + 2\log_2\left(\frac{L}{mc_2}\right) + \max\left(2+2\log_{1/\beta}\left(\frac{-d^T\nabla f(x)}{mT\norm{d}^2}\right), \log_2\left(\frac{-TL\norm{d}^2}{d^T\nabla f(x)}\right)\right)$ function evaluations for a line search in direction $d$ beginning with trial step size $T$ (Lemma~\ref{Wolfe function complexity}). AELS requires no more than $2 + \max\left(\log_{1/\beta}\left(\frac{-\beta(1+\sqrt{1-m/L})d^T\nabla f(x)}{mT\norm{d}^2}\right), \log_{1/\beta}\left(\frac{-mT\norm{d}^2}{\beta^4(1-\sqrt{1-m/L})d^T\nabla f(x)}\right)\right)$ function evaluations for a line search in direction $d$ beginning with trial step size $T$ (Theorem~\ref{approx exact strongly convex function rate}). 

In the special case when $m=L$ (the objective is perfectly conditioned) and $d = -\nabla f(x)$ (the search direction is the negative gradient, steepest descent), strong Wolfe line search uses no more than $2 + 2\log_2\left(\frac{1}{c_2}\right)+\max\left(2 + 2\log_{1/\beta}\left(\frac{1}{TL}\right), \log_2\left(TL\right)\right)$ function evaluations per line search, whereas AELS requires no more than $1 + \max\left(\log_{1/\beta}\left(\frac{1}{TL}\right), 5 + \log_{1/\beta}\left(TL\right) \right)$ function evaluations per line search, where both are initialized with a starting step size of $T$. These rates are quite similar, differing only in the distribution of constants between the two cases (which correspond to whether the initial step size is too small or too large), and the base of the logarithm which is a design choice in each algorithm.

\section{Proofs of theoretical results}
\label{sec:proofs}

We now prove the bounds in Section~\ref{sec:theory}. 

\subsection{Convergence rates for general Armijo line searches}

\begin{proof}[Proof of Lemma~\ref{general armijo strongly convex rate}]
The convergence rates using the true gradient and using a finite differencing approximation both follow (via Corollary~\ref{directional derivative corollary}) from the more general rate: 
If $\norm{g(x) - \nabla f(x)} \leq \rho$, then
\begin{align}
f(x_{k+1}) - f_* \leq \min \Big\{& (1-m\tmin)^k(f(x_0) - f_*) + \frac{\tmax \rho L \norm{x_0 - x_*}}{m\tmin}, \nonumber \\&(1-\frac{m\tmin}{2})^k(f(x_0) - f_*) + \frac{\rho^2 \tmax}{m\tmin} \Big\}\;. 
\end{align}

We prove this more general rate by combining two bounds, the first of which follows a technique based on chapter 3 of \cite{recht-wright}.
\begin{align}
    f(x_{k+1}) - f_* &= f(x_k - t_k g(x_k)) - f_* \nonumber \\
    &\leq f(x_k) -f_* - \frac{t_k}{2}\norm{g(x_k)}^2 ~~\text{(since $t_k$ satisfies Armijo)} \nonumber \\
    &= f(x_k) -f_* - \frac{t_k}{2}\norm{\nabla f(x_k)}^2 - t_k(g(x_k) - \nabla f(x_k))^T\nabla f(x_k) \nonumber \\
    &~~~~~~~~ - \frac{t_k}{2}\norm{g(x_k) - \nabla f(x_k)}^2 \nonumber \\
    &\leq f(x_k) -f_* - mt_k(f(x_k) - f_*) - t_k(g(x_k) - \nabla f(x_k))^T\nabla f(x_k) \nonumber \\
    &\leq (1-m t_k)(f(x_k) - f_*) + t_k \rho \norm{\nabla f(x_k)} \nonumber \\
    &\leq (1 - m \tmin)(f(x_k) - f_*) + \tmax \rho L \norm{x_k - x_*} ~~\text{(by smoothness)}\;,
\end{align}
where the third to last line follows from strong convexity (Lemma~\ref{strongly convex property}). Then,
\begin{align}
    f(x_k) - f_* &\leq (1-m\tmin)^k(f(x_0) - f_*) + \sum_{i=0}^{k-1}(1-m\tmin)^i \tmax \rho L \norm{x_i - x_*} \nonumber \\
    &\leq (1-m\tmin)^k(f(x_0) - f_*) + \frac{\tmax \rho L \norm{x_0 - x_*}}{m\tmin}(1-(1-m\tmin)^k)\;,
\end{align}
where the last term captures the error introduced by approximating the gradient. This error term is bounded for any fixed $k$, and if $\tmin \in (0, \frac{2}{m})$, we can bound the last term by its limit as $k \to \infty$:
\begin{align}
    f(x_k) - f_* &\leq (1-m\tmin)^k(f(x_0) - f_*) + \frac{\tmax \rho L \norm{x_0 - x_*}}{m\tmin}\;.
\end{align}

Alternatively,
\begin{align}
    f(x_{k+1}) - f_* &= f(x_k - t_k g(x_k)) - f_* \nonumber \\
    &\leq f(x_k) -f_* - \frac{t_k}{2}\norm{g(x_k)}^2 ~~\text{(since $t_k$ satisfies Armijo)} \nonumber \\
    &\leq f(x_k) -f_* + \frac{t_k}{2}(\rho^2 - \frac{1}{2}\norm{\nabla f(x_k)}^2) \nonumber \\
    &\leq f(x_k) -f_* + \frac{t_k}{2}(\rho^2 - m(f(x_k) - f_*)) \nonumber \\
    &= (1-\frac{mt_k}{2})(f(x_k) - f_*) + \frac{\rho^2 t_k}{2} \nonumber \\
    &\leq (1-\frac{m\tmin}{2})(f(x_k) - f_*) + \frac{\rho^2 \tmax}{2} \nonumber \\
    &\leq (1-\frac{m\tmin}{2})^k(f(x_0) - f_*) + \frac{\rho^2 \tmax}{m\tmin}\;.
\end{align}
Combining these, we have
\begin{align}
f(x_{k+1}) - f_* \leq \min \Big\{& (1-m\tmin)^k(f(x_0) - f_*) + \frac{\tmax \rho L \norm{x_0 - x_*}}{m\tmin}, \nonumber \\&(1-\frac{m\tmin}{2})^k(f(x_0) - f_*) + \frac{\rho^2 \tmax}{m\tmin} \Big\}\;. 
\end{align}
We expect the first bound to be active early in optimization, when the error is large but decreasing rapidly, and the second bound to be active later in optimization, once the error is smaller and decreasing slightly more slowly.

Next, we prove the convergence rate for the random direction case, where $d =  (-f'_{v_k}(x_k) + \rho_k)v_k$ for an arbitrary unit vector $v_k$ and $\absval{\rho_k} \leq \frac{L\sigma}{2}$ (\ie $\norm{d}$ is a $\sigma$-finite differencing approximation of the directional derivative magnitude along $v_k$, using the bound of Lemma~\ref{directional derivative approximation}). Again, we combine two bounds.

Let $t_k$ be the step size used in step $k$, so $x_{k+1} = x_k + t_k (-f'_{v_k}(x_k) + \rho_k) v_k$. Using the approximated Armijo condition:
\begin{align}
f(x_{k+1}) &\leq f(x_k) - \frac{t_k}{2} (-f'_{v_k}(x_k) + \rho_k)^2 \nonumber \\
&= f(x_k) - \frac{t_k}{2} (f'_{v_k} (x_k))^2 + t_k\rho_k f'_{v_k}(x_k) - \frac{t_k}{2}\rho_k^2 \nonumber \\
&\leq f(x_k) - \frac{t_k}{2} (f'_{v_k} (x_k))^2 + t_k\rho_k f'_{v_k}(x_k) \nonumber \\
&\leq f(x_k) - \frac{t_k}{2} (f'_{v_k} (x_k))^2 + t_k\frac{L\sigma}{2}\absval{ f'_{v_k}(x_k)} \nonumber\\
&\leq f(x_k) - \frac{\tmin}{2} (f'_{v_k} (x_k))^2 + \tmax\frac{L^2\sigma\norm{x_k - x_*}}{2}\;,
\end{align}
where we used Lemma~\ref{directional derivative approximation} and Lipschitz gradient to bound $\rho_k f'_{d_k}(x_k)$. Alternatively, 
\begin{align}
f(x_{k+1}) &\leq f(x_k) - \frac{t_k}{2} (-f'_{v_k}(x_k) + \rho_k)^2 \nonumber \\
&\leq f(x_k) - \frac{t_k}{4} (f'_{v_k} (x_k))^2 + \frac{t_k}{2}\rho_k^2 \nonumber \\
&\leq f(x_k) - \frac{\tmin}{4} (f'_{v_k} (x_k))^2 + \frac{\tmax L^2\sigma^2}{8} \;.
\end{align}
Since both of these bounds must be satisfied, we have
\begin{align}
f(x_{k+1}) \leq f(x_k) - \max \Big\{& \frac{\tmin}{2} (f'_{v_k} (x_k))^2 - \frac{\tmax L^2\sigma\norm{x_k - x_*}}{2} \nonumber \\
& \frac{\tmin}{4} (f'_{v_k} (x_k))^2 - \frac{\tmax L^2\sigma^2}{8}  \Big\} \;.
\end{align}

Now, we take expectations over all the random directions $v_0, ..., v_k$:
\begin{align}
\E[ (f'_{v_k}(x_k))^2 ] &= \E\left[ \E \left[ (f'_{v_k}(x_k))^2 \mid v_0, ..., v_{k-1} \right] \right] \nonumber \\
&= \E\left[ \E \left[ (v_k^T \nabla f(x_k))^2 \mid v_0, ..., v_{k-1} \right] \right] \nonumber \\
&= \E\left[ (\nabla f(x_k))^T \E \left[ v_k v_k^T\right] \nabla f(x_k)  \right] \nonumber \\
&= \frac{1}{n}\E\left[ \norm{ \nabla f(x_k)}^2  \right] \;,
\end{align}
where $n$ is the dimension of the space and $v_k = \frac{z}{\norm{z}}$ for $z \sim \mathcal{N}(0, I_n)$.
Now our bound is: 
\begin{align}
\E[f(x_{k+1})] \leq \E[f(x_k)] - \max \Big\{& \frac{\tmin}{2n} \E\left[ \norm{ \nabla f(x_k)}^2  \right] - \frac{\tmax L^2\sigma\norm{x_k - x_*}}{2} \nonumber \\
& \frac{\tmin}{4n} \E\left[ \norm{ \nabla f(x_k)}^2  \right] - \frac{\tmax L^2\sigma^2}{8}  \Big\} \;.
\end{align}

Subtracting $f_*$ from each side, we have:
\begin{align}
\E[f(x_{k+1}) - f_*] &\leq \E[f(x_k) - f_*] - \max \Big\{\frac{\tmin}{2n} \E\left[ \norm{ \nabla f(x_k)}^2  \right] - \frac{\tmax L^2\sigma\norm{x_k - x_*}}{2}, \nonumber \\
&~\hspace{3.8cm} \frac{\tmin}{4n} \E\left[ \norm{ \nabla f(x_k)}^2  \right] - \frac{\tmax L^2\sigma^2}{8}  \Big\} \nonumber \\
&\leq \min \Big\{\E[f(x_k) - f_*]\Big(1 - \frac{m\tmin}{n} \Big) + \frac{\tmax L^2\sigma\norm{x_k - x_*}}{2}, \nonumber   \\
&~\hspace{1.2cm} \E[f(x_k) - f_*]\Big(1 - \frac{m\tmin}{2n} \Big) + \frac{\tmax L^2\sigma^2}{8} \Big\} \;.
\end{align}
where the last inequality follows from strong convexity (Lemma~\ref{strongly convex property}). Repeating over the $k$ steps, we have nearly linear convergence:
\begin{align}
& \hspace{-3cm}\E[f(x_k) - f_*] \leq \nonumber \\
\min\Big\{(f(x_0)-f_*)& \Big(1- \frac{m \tmin}{n} \Big)^k + \frac{\sigma L^2 \tmax n\norm{x_0 - x_*}}{2m \tmin}\Big(1- \left( 1-\frac{m \tmin}{n} \Big)^k \right), \nonumber \\
 (f(x_0)-f_*) &\Big(1- \frac{m \tmin}{2n} \Big)^k + \frac{\tmax L^2\sigma^2n}{4m\tmin}\Big(1- \Big( 1-\frac{m \tmin}{2n} \Big)^k \Big)  \Big\} \;,
\end{align}
where the second term in each bound captures the error introduced by approximating the directional derivative by finite differencing in each step. This error term is bounded for any number of iterations $k$, and since $\tmin \in (0, \frac{2n}{m})$ (because $\tmin$ satisfies the Armijo condition) it converges to $\frac{\sigma L^2 \tmax n \norm{x_0 - x_*}}{2m \tmin}$ or $\frac{\tmax L^2\sigma^2n}{4m\tmin}$, respectively, which can be managed by judicious choice of $\sigma$. As in the gradient case, we expect the first bound to be active towards the beginning of training when error is large, and the second bound to be active later in training when the error is small.

\end{proof}

Lemma~\ref{general armijo strongly convex rate} depends on $\tmin$, which in turn may depend on $t_a$, so we next prove bounds on $t_a$. Lemma~\ref{t_a bounds} will also be useful in the convergence analysis of AELS.

\begin{lemma}
\label{t_a bounds}
If $f$ is $m$-strongly convex with $L$-Lipschitz gradient and $d$ is a search (descent) direction, then
\begin{align}
t_a \in  \frac{-d^T\nabla f(x)}{\norm{d}^2} \left[\frac{1}{L}, \frac{1}{m} \right]\;.
\end{align}
Note that the first inequality requires only the Lipschitz gradient condition, and the second requires only the strong convexity condition. Without strong convexity (as $m \to 0$) we approach the one-sided bound $t_a \geq \frac{-d^T\nabla f(x)}{L\norm{d}^2}$.
\begin{proof}
Applying Lemma~\ref{taylor and lipschitz and strong convexity} to the case $t = t_a$, 
\begin{align}
f(x) + t_a d^T \nabla f(x) + \frac{mt_a^2}{2}\norm{d}^2 \leq f(x) + \frac{t_a}{2}d^T\nabla f(x) \leq f(x) + t_a d^T \nabla f(x) + \frac{Lt_a^2}{2}\norm{d}^2\;.\hspace{-0.1in}
\end{align}
Rearranging, we have:
\begin{align}
\frac{mt_a}{2}\norm{d}^2 \leq  \frac{-1}{2}d^T\nabla f(x) \leq \frac{Lt_a}{2}\norm{d}^2\;,
\end{align}
which implies the lemma.
\end{proof}
\end{lemma}

Lemma~\ref{t_a bounds} gives upper and lower bounds on $t_a$ as a function of $\frac{d^T\nabla f(x)}{\norm{d}^2}$, which varies depending on the local behavior of the objective and the method for choosing a descent direction $d$. Accordingly, the next step is to bound $\frac{d^T\nabla f(x)}{\norm{d}^2}$ for each of the descent directions we consider.

\begin{lemma}
\label{gamma bounds}
If we set $d = -f'_v(x)v$ for an arbitrary unit vector $v$, including the case where $v = \frac{\nabla f(x)}{\norm{\nabla f(x)}}$ so $d = -\nabla f(x)$, then
\begin{align}
\frac{-d^T\nabla f(x)}{\norm{d}^2} &= 1\;.
\end{align}

If we set $d = -\nabla f(x) + \rho$ where $\norm{\rho} \leq \frac{L\sigma\sqrt n}{2}$ (\ie $d$ is a $\sigma$-finite differencing approximation of $-\nabla f(x)$, using the bound of Corollary~\ref{directional derivative corollary}), with $\sigma = \frac{2\tilde \sigma m \Delta}{L\sqrt n}$ and $\tilde \sigma < 1$, then
\begin{align}
\frac{-d^T\nabla f(x)}{\norm{d}^2} &\in \left[ \frac{1-\tilde\sigma}{(1+\tilde\sigma)^2}, \frac{1}{1-\tilde\sigma} \right]\;.
\end{align}

If we set $d = (-f'_v(x) + \rho)v$ for an arbitrary unit vector $v$ where $\absval{\rho} \leq \frac{L\sigma}{2}$ (\ie $\norm{d}$ is a $\sigma$-finite differencing approximation of the directional derivative magnitude along $v$, using the bound of Lemma~\ref{directional derivative approximation}), with $\sigma = \frac{2\tilde\sigma m\Delta}{L}$ and $\tilde\sigma < 1$, then
\begin{align}
\frac{-d^T\nabla f(x)}{\norm{d}^2} &\in \left[ \frac{1}{1+\tilde\sigma}, \frac{1}{1-\tilde\sigma} \right]\;.
\end{align}

\begin{proof}
If we set $d = -f'_v(x)v$ for an arbitrary unit vector $v$, 
then
\begin{align}
\frac{-d^T\nabla f(x)}{\norm{d}^2} &= \frac{f'_v(x) v^T\nabla f(x)}{(f'_v(x))^2} = 1\;.
\end{align}

If we set $d = -\nabla f(x) + \rho$ where $\norm{\rho} \leq \frac{L\sigma\sqrt n}{2}$ 
, with $\sigma = \frac{2\tilde \sigma m \Delta}{L\sqrt n}$ and $\tilde \sigma < 1$, then
\begin{align}
\frac{-d^T\nabla f(x)}{\norm{d}^2} &= \frac{(\nabla f(x) - \rho)^T\nabla f(x)}{\norm{\nabla f(x) - \rho}^2} \nonumber \\
&\in \left[\frac{\norm{\nabla f(x)}^2 - \norm{\nabla f(x)}\norm{\rho}}{\norm{\nabla f(x) - \rho}^2} , \frac{\norm{\nabla f(x)}}{\norm{\nabla f(x) - \rho}} \right] \nonumber \\
&\subseteq \left[ \frac{m\norm{x-x_*}(m\norm{x-x_*} - \tilde\sigma m \Delta)}{(m\norm{x-x_*} + \tilde\sigma m \Delta)^2}, \frac{m\norm{x-x_*}}{m\norm{x-x_*} - \tilde\sigma m \Delta} \right] \nonumber \\
&\subseteq \left[ \frac{1-\frac{\tilde\sigma \Delta}{\norm{x-x_*}}}{(1+\frac{\tilde\sigma\Delta}{\norm{x-x_*}})^2}, \frac{1}{1 - \frac{\tilde\sigma \Delta}{\norm{x-x_*}}} \right] \nonumber \\
&\subseteq \left[ \frac{1-\tilde\sigma}{(1+\tilde\sigma)^2}, \frac{1}{1-\tilde\sigma} \right]\;.
\end{align}

If we set $d = (-f'_v(x) + \rho)v$ for an arbitrary unit vector $v$ where $\absval{\rho} < \frac{L\sigma}{2}$ 
, with $\sigma = \frac{2\tilde\sigma m\Delta}{L}$ and $\tilde\sigma < 1$, then
\begin{align}
\frac{-d^T\nabla f(x)}{\norm{d}^2} &= \frac{(f'_v(x)-\rho)v^T\nabla f(x)}{(f'_v(x)-\rho)^2} = \frac{f'_v(x)}{f'_v(x)-\rho} \nonumber \\
&\in \left[ \frac{m\norm{x-x_*}}{m\norm{x-x_*} + m\tilde\sigma\Delta}, \frac{m\norm{x-x_*}}{m\norm{x-x_*}-m\tilde\sigma\Delta} \right] \nonumber \\
&\subseteq \left[ \frac{1}{1+\frac{\tilde\sigma\Delta}{\norm{x-x_*}}}, \frac{1}{1-\frac{\tilde\sigma\Delta}{\norm{x-x_*}}} \right] \nonumber \\
&\subseteq \left[ \frac{1}{1+\tilde\sigma}, \frac{1}{1-\tilde\sigma} \right]\;.
\end{align}
\end{proof}
\end{lemma}

The constraints on $\sigma$ in Lemma~\ref{gamma bounds} guarantee that we choose a descent direction; in our analysis we assume they are satisfied. Although we commonly don't know $m$ and $L$ in practice, setting $\sigma$ to be a fixed small constant (\eg $10^{-8}$) suffices in our experiments. In the limit as $\sigma \to 0$, the bounds on $\frac{-d^T\nabla f(x)}{\norm{d}^2}$ using finite differencing approach $1$, the value using the true gradient or directional derivative.

\subsection{Approximately exact line search}

\begin{proof}[Proof of Invariant~\ref{approx exact invariant}]
We begin with some notation and observations. 

Let $h(s) = f(x + sd)$ for current iterate $x$ and search direction $d$. If $f$ is strictly unimodal (\eg if $f$ is strongly convex), then $h$ is decreasing for all $s \in [0, t_*)$ and increasing for all $s > t_*$. Therefore, if we can find some $s_1 < s_2 < s_3$ such that $h(s_1) \geq h(s_2)$ and $h(s_2) \leq h(s_3)$, then $t_* \in [s_1, s_3]$. AELS follows this methodology by always returning a step size $t$ for which $h(t) \geq h(\beta^{-1}t)$ and $h(\beta^{-1}t) \leq h(\beta^{-2}t)$.

We analyze AELS (Algorithm~\ref{alg:approximately exact line search with warm start}) step by step, beginning with the first If statement, which compares $h(0)$ and $h(T)$. There are three cases, as illustrated in Figure~\ref{fig:alg_execution_illustration}. 

If $h(T) > h(0)$, then we know $t_* < T$, and accordingly we choose to backtrack (setting $\alpha = \beta < 1$). Note that by convexity, our first backtracking step (with step size $\beta T$) in this scenario will always satisfy $h(\beta T) < h(T)$. We then repeatedly multiply the step size by $\beta$ until we reach a step size $t$ with $h(t) \geq h(\beta^{-1}t)$, at which point we are also guaranteed that $h(\beta^{-1}t) < h(\beta^{-2}t)$. AELS returns this $t$, which satisfies the desired invariant.

If instead $h(T) \leq h(0)$, we cannot be sure the relationship between $t_*$ and $T$, so we try to forward-track (setting $\alpha = \beta^{-1} > 1$). In the first forward-tracking step, we test a step size $\beta^{-1}T$. If $h(\beta^{-1}T) < h(T)$, we continue forward-tracking until we reach a step size $t$ with $h(t) \geq h(\beta t)$, at which point we are also guaranteed that $h(\beta t) < h(\beta^{2}t)$. AELS returns $\beta^2 t$, which satisfies the desired invariant.

The third and final case is if $h(T) \leq h(0)$ but $h(\beta^{-1}T) \geq h(T)$, meaning that we tried forward-tracking for one loop iteration and immediately the objective value increased. We then follow the second If statement, reset the step size to $T$, and know that $h(T) < h(\beta^{-1}T)$. We backtrack by factors of $\beta$ until we reach a step size $t$ with $h(t) > h(\beta^{-1}t)$, at which point we are also guaranteed that $h(\beta^{-1}t) \leq h(\beta^{-2}t)$. AELS returns this $t$, which satisfies the desired invariant.
\end{proof}

To bound the iteration complexity of AELS, we must first bound $t_*$.

\begin{lemma}
\label{t_* bounds}
If $f$ is $m$-strongly convex with $L$-Lipschitz gradient and $d$ is a descent direction, then
\begin{align}
t_* \in \frac{-d^T\nabla f(x)}{m\norm{d}^2} \left[1-\sqrt{1-\frac{m}{L}}, 1+\sqrt{1-\frac{m}{L}}  \right]\;.
\end{align}
Note that without strong convexity (as $m \to 0$) we approach the one-sided bound $t_* \geq \frac{-d^T\nabla f(x)}{2L\norm{d}^2}$.
\begin{proof}
By optimality of $t_*$,
\begin{align}
    f(x + t_* d) &\leq f\left(x - \frac{d^T\nabla f(x)}{L\norm{d}^2} d\right) \nonumber \\
    &\leq f(x) - \frac{(d^T\nabla f(x))^2}{2L\norm{d}^2}\;,
\end{align}
where the last line follows from Lemma~\ref{t_a bounds}.
Combining this with strong convexity (Definition~\ref{strong convexity}),
\begin{align}
f(x) + t_* d^T\nabla f(x) + \frac{mt_*^2}{2}\norm{d}^2 \leq f(x) - \frac{(d^T\nabla f(x))^2}{2L\norm{d}^2}\;,
\end{align}
which implies that
$
\frac{mt_*^2}{2}\norm{d}^2 + t_* d^T\nabla f(x) + \frac{(d^T\nabla f(x))^2}{2L\norm{d}^2}\leq 0.
$
This is satisfied for $t_*$ in the interval $\frac{-d^T\nabla f(x)}{m\norm{d}^2} \left[1-\sqrt{1-\frac{m}{L}}, 1+\sqrt{1-\frac{m}{L}}  \right]$.
\end{proof}
\end{lemma}

We can now prove Theorem~\ref{approx exact strongly convex iteration rate}, which bounds the iteration complexity of AELS.

\begin{proof}[Proof of Theorem~\ref{approx exact strongly convex iteration rate}]
Let $f$ be $m$-strongly convex with $L$-Lipschitz gradient. Let $x_{k+1} = x_k - t_k g(x_k)$, where $t_k$ is chosen according to Algorithm~\ref{alg:approximately exact line search with warm start} with search direction $d_k = -g(x_k)$. Let $\gamma_k = \frac{g(x_k)^T\nabla f(x_k)}{\norm{g(x_k)}^2}$.

We begin with the case where $\norm{g(x) - \nabla f(x)} \leq \rho$, which occurs when either $g(x_k)$ is the true gradient $\nabla f(x_k)$ or when $g(x_k)$ is a finite differencing (or other sufficiently accurate) approximation of the gradient. Let $t_a$ be the step size that satisfies the Armijo condition with equality (with gradient approximation $g(x_k)$ and direction $-g(x_k)$) and $t_*$ be the exact line search step size (in direction $-g(x_k)$). We begin with the observation that if $t_k > t_a$, it must be because $f(x_k - t_k g(x_k)) \leq f(x_k - t_a g(x_k))$. Then,
\begin{align}
f(x_{k+1}) - f_* &\leq f(x_k) - f_* - \frac{\min(t_k, t_a)}{2}\norm{g(x_k)}^2 \;,
\end{align}
and we can follow exactly the same steps as in the proof of Lemma~\ref{general armijo strongly convex rate}, replacing $t_k$ with $\min(t_k, t_a)$. By combining Lemma~\ref{t_a bounds}, Invariant~\ref{approx exact invariant}, and Lemma~\ref{t_* bounds}, we can bound $\min(t_k, t_a) \in \gamma_k\left[ \frac{\beta^2}{m}(1-\sqrt{1-\frac{m}{L}}), \frac{1}{m}\right]$, with $\gamma_k \in [\gammamin, \gammamax]$ according to Lemma~\ref{gamma bounds}. Following this through, we have:

\begin{align}
& \hspace{-1cm} f(x_{k+1}) - f_* \leq \nonumber \\
\min \Big\{& \Big(1-\gammamin\beta^2\Big(1-\sqrt{1-\frac{m}{L}}\Big)\Big)^k(f(x_0) - f_*) + \frac{\gammamax \rho L \norm{x_0 - x_*}}{m\gammamin\beta^2\Big(1-\sqrt{1-\frac{m}{L}}\Big)}, \nonumber \\
&\Big(1-\frac{1}{2}\gammamin\beta^2\Big(1-\sqrt{1-\frac{m}{L}}\Big)\Big)^k(f(x_0) - f_*) + \frac{\rho^2 \gammamax}{m\gammamin\beta^2\Big(1-\sqrt{1-\frac{m}{L}}\Big)} \Big\}\;. 
\end{align}

Combining this bound with Lemma~\ref{gamma bounds} and Corollary~\ref{directional derivative corollary} yields the stated bounds using the true gradient and a finite differencing approximation. 

Next, we prove the bound for the case where $g(x_k) = \mu_k v_k$, where $\mu_k$ is a finite differencing approximation of $f'_{v_k}(x_k)$ for a uniformly random unit vector $v_k$. For Algorithm~\ref{alg:approximately exact line search with warm start}, all we need to assume about $\mu_k$ is that it has the same sign as $f'_{v_k}(x_k)$, so that $-g(x_k)$ is a descent direction. Let $t_k$ be the step size chosen in step $k$, so $x_{k+1} = x_k - t_k \mu_k v_k$. We begin by introducing some auxiliary variables:
\begin{itemize}
\item $s_k = t_k\absval{\mu_k}$, so $x_{k+1} = x_k - s_k v_k \text{sign}(f'_{v_k}(x_k))$. 

\item $s_a = t_a\absval{f'_{v_k}(x_k)}$, where $t_a$ satisfies the Armijo condition (Definition~\ref{armijo}) with equality (for $d = -f'_{v_k}(x_k)v_k$), so $f(x_k - s_a v_k \text{sign}(f'_{v_k}(x_k)) = f(x_k) - \frac{s_a}{2}\absval{f'_{v_k}(x_k)}$. Since by Lemma~\ref{t_a bounds} $t_a \in [\frac{1}{L}, \frac{1}{m}]$, $s_a \in [\frac{1}{L}\absval{f'_{v_k}(x_k)}, \frac{1}{m}\absval{f'_{v_k}(x_k)}]$.

\item $s_* = t_*\absval{f'_{v_k}(x_k)}$, where $t_*$ is an exact line search step size (along direction $d = -f'_{v_k}(x_k)v_k$), so $f(x_k - s_* v_k \text{sign}(f'_{v_k}(x_k)) = \min_s f(x_k - s v_k \text{sign}(f'_{v_k}(x_k))$. Since by Lemma \ref{t_* bounds}~ $t_* \in \left[\frac{1}{m}\left(1-\sqrt{1-\frac{m}{L}}\right), \frac{1}{m}\left(1+\sqrt{1-\frac{m}{L}}\right) \right]$, we have $s_* \in \left[\frac{1}{m}\left(1-\sqrt{1-\frac{m}{L}}\right)\absval{f'_{v_k}(x_k)}, \frac{1}{m}\left(1+\sqrt{1-\frac{m}{L}}\right)\absval{f'_{v_k}(x_k)} \right]$.
\end{itemize}
Rewriting Invariant~\ref{approx exact invariant} in this notation, we have that $s_k \in [\beta^2 s_*, s_*]$. One of the following cases must hold:
\begin{itemize}
\item $s_k \leq s_a$: Then $f(x_k - s_k v_k \text{sign}(f'_{v_k}(x_k))) \leq f(x_k) - \frac{s_k}{2}\absval{f'_{v_k}(x_k)}$.
\item $s_k > s_a$: Then $f(x_k - s_k v_k \text{sign}(f'_{v_k}(x_k))) \leq f(x_k) - \frac{s_a}{2}\absval{f'_{v_k}(x_k)}$.
\end{itemize}
Combining these cases, we have:

\begin{align}
f(x_{k+1}) 
&\leq f(x_k) - \frac{\min(s_k, s_a)}{2}\absval{f'_{v_k}(x_k)} \nonumber \\
&\leq f(x_k) - \frac{1}{2}\min\left(\frac{\beta^2}{m}\left(1-\sqrt{1-\frac{m}{L}}\right)\absval{f'_{v_k}(x_k)}, \frac{1}{L}\absval{f'_{v_k}(x_k)} \right) \absval{f'_{v_k}(x_k)} \nonumber \\
&= f(x_k) - \frac{1}{2}\min\left(\frac{\beta^2}{m}\left(1-\sqrt{1-\frac{m}{L}}\right), \frac{1}{L}\right) (f'_{v_k}(x_k))^2 \nonumber \\
&\leq f(x_k) - \frac{\beta^2}{2m}\left(1-\sqrt{1-\frac{m}{L}}\right) (f'_{v_k}(x_k))^2\;,
\end{align}
where we plugged in our bounds on $s_k$ and $s_a$ and then simplified by noticing that for any $m \in (0, L]$,  $\frac{\beta^2}{m}\left(1-\sqrt{1-\frac{m}{L}}\right) \leq \frac{1}{L}$.
Now, we take expectations over all the random directions $v_0, ..., v_k$ and follow the same steps as in the proof of the random direction portion of Lemma~\ref{general armijo strongly convex rate}:
\begin{align}
\E[f(x_{k+1})] &\leq \E[f(x_k)] - \frac{\beta^2}{2m}\left(1-\sqrt{1-\frac{m}{L}}\right)\E[ (f'_{v_k}(x_k))^2 ] \nonumber \\
&= \E[f(x_k)] - \frac{\beta^2}{2mn}\left(1-\sqrt{1-\frac{m}{L}}\right) \E\left[ \norm{ \nabla f(x_k)}^2  \right]\;,
\end{align}
where $n$ is the dimension of the space and $v_k = \frac{z}{\norm{z}}$ for $z \sim \mathcal{N}(0, I_n)$. Subtracting $f_*$ from each side, we have:
\begin{align}
\E[f(x_{k+1}) - f_*] &\leq \E[f(x_k) - f_*] - \frac{\beta^2}{2mn}\left(1-\sqrt{1-\frac{m}{L}}\right) \E\left[ \norm{ \nabla f(x_k)}^2  \right] \nonumber \\
&\leq \E[f(x_k) - f_*] \left(1- \frac{\beta^2}{n}\left(1-\sqrt{1-\frac{m}{L}}\right)  \right)\;,
\end{align}
where the last inequality follows from strong convexity (Lemma~\ref{strongly convex property}). Repeating over the $k$ steps, we have linear convergence:
\begin{align}
\E[f(x_k) - f_*] &\leq (f(x_0)-f_*) \left(1- \frac{\beta^2}{n}\left(1-\sqrt{1-\frac{m}{L}}\right)  \right) ^k\;.
\end{align}
\end{proof}

Finally, we prove Theorem~\ref{approx exact strongly convex function rate}, which bounds the function evaluation complexity of AELS (Algorithm~\ref{alg:approximately exact line search with warm start}).

\begin{proof}[Proof of Theorem~\ref{approx exact strongly convex function rate}]
Let $t_k$ denote the step size taken in step $k$, where by convention we set $t_{-1} = T_0$. The initial step size in step $k$ is then $\beta^{-1}t_{k-1}$. We ignore the initial evaluation of $f(x_k)$ in each step as this can be reused between steps (except for the first step). In step $k$, AELS first makes one function evaluation to compute $f(x_k + \beta^{-1}t_{k-1} d_k)$. Then, there are three cases (as in the proof of Invariant~\ref{approx exact invariant}, whose $h$ notation we reuse). 

If $h(\beta^{-1}t_{k-1}) > h(0)$, then we repeatedly multiply the step size by $\beta$ (making one function evaluation for each factor of $\beta$) until we reach a step size $t$ with $t \in [\beta^2 t_*, t_*]$, which is used as $t_k$. If $p$ is the number of function evaluations in this loop, then $t_k = \beta^{p-1}t_{k-1}$, so $p \leq 2 + \ceil*{\log_{1/\beta}\left(\frac{t_{k-1}}{t_*}\right)}$. 

If instead $h(\beta^{-1}t_{k-1}) \leq h(0)$, we try to forward-track (setting $\alpha = \beta^{-1} > 1$). In the first forward-tracking step, we test a step size $\beta^{-2}t_{k-1}$. If $h(\beta^{-2}t_{k-1}) < h(\beta^{-1}t_{k-1})$, we continue multiplying the step size by $\beta^{-1}$ (making one function evaluation for each factor of $\beta^{-1}$) until we reach a step size $t$ with $t \in [t_*, \beta^{-2} t_*]$, and then use $t_k = \beta^2 t$. If $q$ is the number of function evaluations in this loop, then $t_k = \beta^{-q+1}t_{k-1}$, so $q \leq \ceil*{\log_{1/\beta}\left(\frac{t_*}{t_{k-1}}\right)}$.

The third and final case is if $h(\beta^{-1}t_{k-1}) \leq h(0)$ but $h(\beta^{-2}t_{k-1}) \geq h(\beta^{-1}t_{k-1})$, meaning that we tried forward-tracking for one loop iteration and immediately the objective value increased. We then follow the second If statement, reset the step size to $\beta^{-1}t_{k-1}$, and know that $h(\beta^{-1}t_{k-1}) \leq h(\beta^{-2}t_{k-1})$. We backtrack by factors of $\beta$ until we reach a step size $t$ with $t \in [\beta^2 t_*, t_*]$, which is used as $t_k$. We used one function evaluation to evaluate $h(\beta^{-2}t_{k-1})$, and then $r$ function evaluations to backtrack (one per loop iteration). Then $t_k = \beta^{r-1}t_{k-1}$, so $r \leq 2 + \ceil*{\log_{1/\beta}\left(\frac{t_{k-1}}{t_*}\right)}$ and the number of additional function evaluations is $\leq 3 + \ceil*{\log_{1/\beta}\left(\frac{t_{k-1}}{t_*}\right)}$.

Combining these cases, the number of function evaluations in step $k$ is
\begin{align}
n_k &\leq 1 + \max\left(\ceil*{\log_{1/\beta}\left(\frac{t_*}{t_{k-1}}\right)}, 3 + \ceil*{\log_{1/\beta}\left(\frac{t_{k-1}}{t_*}\right)}\right)\;.
\end{align}

For $k=0$, $t_{k-1} = T_0$, so
\begin{align}
n_0 &\leq 1 + \ceil*{\log_{1/\beta}\max\left(\frac{\gamma_0(1+\sqrt{1-\frac{m}{L}})}{mT_0}, \frac{mT_0}{\beta^3\gamma_0(1-\sqrt{1-\frac{m}{L}})}\right)}\;.
\end{align}

For $k > 0$, $t_{k-1} \in \frac{\gamma_{k-1}}{m}\left[\beta^2\left(1-\sqrt{1-\frac{m}{L}}\right), 1+\sqrt{1-\frac{m}{L}} \right]$ (Lemma~\ref{t_* bounds}), so
\begin{align}
n_k &\leq 4 + \ceil*{\log_{1/\beta}\left(\frac{\gammamax(1+\sqrt{1-\frac{m}{L}})}{\gammamin(1-\sqrt{1-\frac{m}{L}})}\right)}\;.
\end{align}

We can now bound the total number of function evaluations $N_k$ in the first $k$ steps:
\begin{align}
N_k &= 1 + n_0 + \sum_{i=1}^{k-1} n_k \nonumber \\
&\leq 2 + \ceil*{\log_{1/\beta}\max\left(\frac{\gamma_0(1+\sqrt{1-\frac{m}{L}})}{mT_0}, \frac{mT_0}{\beta^3\gamma_0(1-\sqrt{1-\frac{m}{L}})}\right)} \nonumber \\
&~~~~~~~~ + (k-1)\left(4 + \ceil*{\log_{1/\beta}\left(\frac{\gammamax(1+\sqrt{1-\frac{m}{L}})}{\gammamin(1-\sqrt{1-\frac{m}{L}})}\right)} \right) \nonumber \\
&\leq k\left(4 + \ceil*{\log_{1/\beta}\left(\frac{\gammamax(1+\sqrt{1-\frac{m}{L}})}{\gammamin(1-\sqrt{1-\frac{m}{L}})}\right)} \right) \nonumber \\
&~~~~~~~~ + \ceil*{\log_{1/\beta}\max\left(\frac{\beta^2\gammamin(1-\sqrt{1-\frac{m}{L}})}{mT_0}, \frac{mT_0}{\beta\gammamax(1+\sqrt{1-\frac{m}{L}})}\right)}\;.
\end{align}
\end{proof}

\subsection{Strong Wolfe line search}
\label{sec:wolfe_feval}

We bound the number of function (and potentially directional derivative) evaluations used by the strong Wolfe condition line search presented in Algorithms 3.5 and 3.6 of Nocedal and Wright \cite{nocedal2006line}, with the simplification that successive trial step lengths in Algorithm 3.5 are chosen by scaling by a factor $\beta^{-1}$, where $\beta \in (0, 1)$, and that interpolation in Algorithm 3.6 is done by bisection. 

We reproduce these slightly simplified algorithms here for convenience, using the notation that $h$ is the objective function $f$ restricted to the line, and $\alpha_i$ is the candidate step size in iteration $i$ of Algorithm~\ref{alg:wolfe}.
\SetKwFor{Loop}{Loop}{}{EndLoop}
\SetKwFor{Break}{Break}{}{}
\noindent
\begin{algorithm2e}[h]
\DontPrintSemicolon
\KwIn{$h, \alpha_0=0, \alpha_1>0, \beta \in (0,1), c_1 \in(0,1), c_2 \in (c_1,1)$}
$i \gets 1$\;
\Loop{}{
 Evaluate $h(\alpha_i)$\;
 \If{$h(\alpha_i) > h(0) + c_1\alpha_ih'(0)$ or $[h(\alpha_i) \geq h(\alpha_{i-1})$ and $i>1]$}{
 	$\alpha_* \gets$~\textbf{Zoom}$(\alpha_{i-1}, \alpha_i)$\;
	\Break{}{}
 }
 Evaluate $h'(\alpha_i)$\;
 \If{$\vert h'(\alpha_i)\vert \leq -c_2h'(0)$}{
 	$\alpha_* \gets \alpha_i$\;
	\Break{}{}
 }
 \If{$h'(\alpha_i) \geq 0$}{
 	$\alpha_* \gets$~\textbf{Zoom}$( \alpha_i, \alpha_{i-1})$\;
	\Break{}{}
 }
 $\alpha_{i+1} \gets \beta^{-1}\alpha_i$\;
 $i \gets i+1$\;
}
\Return{$\alpha_*$}
\caption{Strong Wolfe Line Search: from Algorithm 3.5 in \cite{nocedal2006line}}
\label{alg:wolfe}
\end{algorithm2e}

\noindent
\begin{algorithm2e}[h]
\DontPrintSemicolon
\KwIn{$\alpha_{\text{lo}}, \alpha_{\text{hi}}$}
\Loop{}{
 $\alpha_j \gets (\alpha_{\text{lo}} + \alpha_{\text{hi}})/2$\;
 Evaluate $h(\alpha_i)$\;
 \If{$h(\alpha_j) > h(0) + c_1\alpha_j h'(0)$ or $h(\alpha_j) \geq h(\alpha_{\text{lo}})$}{
 	$\alpha_{\text{hi}} \gets \alpha_j$\;
 }
 \Else{
 Evaluate $h'(\alpha_j)$\;
 \If{$\vert h'(\alpha_j)\vert \leq -c_2h'(0)$}{
 	$\alpha_* \gets \alpha_j$\;
	\Break{}{}
 }
 \If{$h'(\alpha_j)(\alpha_{\text{hi}}-\alpha_{\text{lo}}) \geq 0$}{
 	$\alpha_{\text{hi}} \gets \alpha_{\text{lo}}$\;
 }
 $\alpha_{\text{lo}} \gets \alpha_j$\;
 }
}
\Return{$\alpha_*$}
\caption{Zoom: from Algorithm 3.6 in \cite{nocedal2006line}}
\label{alg:zoom}
\end{algorithm2e}

\begin{proof}[Proof of Lemma~\ref{Wolfe function complexity}]

We begin by bounding the number of iterations of Algorithm~\ref{alg:wolfe}, the outer line search function. 
We use the same $h$ notation as in the proof of Invariant~\ref{approx exact invariant} and Theorem~\ref{approx exact strongly convex function rate}.
We also borrow the notation of \cite{nocedal2006line}, where $\alpha_i$ denotes the candidate step size in iteration $i$ of the main loop in Algorithm~\ref{alg:wolfe}, so $\alpha_1=T$, the initial step size. 

Let $p$ be the number of times we enter the main loop of Algorithm~\ref{alg:wolfe}, such that on the $p$th iteration we either terminate successfully or call Algorithm~\ref{alg:zoom}. Then $\alpha_p = \beta^{1-p}\alpha_1$, and we know that $\alpha_{p-1}$ violates the strong Wolfe curvature condition. Then either $p = 1$ or  $\alpha_{p-1} = \beta^{2-p}\alpha_1 < t_w \leq \frac{-d^T\nabla f(x)}{m\norm{d}^2}$, so $p \leq \max(2 + \log_{1/\beta}\left(-\frac{d^T\nabla f(x)}{m\alpha_1 \norm{d}^2} \right), 1)$. Each iteration of this outer loop requires one function evaluation and one directional derivative evaluation, the latter of which may be averted on the last iteration if $\alpha_p$ is sufficiently large as to either violate the Armijo sufficient decrease condition or exceed the true line search minimizer. 

Finally, we must bound the number of function and directional derivative evaluations in Algorithm~\ref{alg:zoom}. Let $\alpha_L$ and $\alpha_R$ be the initial parameters passed to Algorithm~\ref{alg:zoom}, denoting the left and right endpoints of the active search window for $\alpha$ (these are renamings of the stated inputs $\alpha_{\text{lo}}$ and $\alpha_{\text{hi}}$, which denote these search endpoints ordered by their objective values). We consider Algorithm~\ref{alg:zoom} in two stages; in the first stage, the first If is repeatedly followed until the right endpoint of the search window satisfies the Armijo sufficient decrease condition. Let $q$ be the number of iterations in this first stage. Then $q < \max \left(0, \log_2\left( \frac{\alpha_R - \alpha_L}{t_a - \alpha_L}\right) \right)$, where $t_a$ is the step size that satisfies the Armijo condition with equality. Note that if $p=1$ then $\alpha_R - \alpha_L = \alpha_1$ and $q > 0$; if $p>1$ then $\alpha_R - \alpha_L = \alpha_1 \left( \beta^{1-p}-\beta^{2-p} \right)$ and $q=0$. Using Lemma~\ref{t_a bounds} we can bound $q < \max \left(0, \log_2\left(\frac{-\alpha_1L\norm{d}^2}{d^T\nabla f(x)}\right)\right)$. Each iteration in this stage requires only a function evaluation (no directional derivatives). 

Once the entire search window satisfies the Armijo condition, Algorithm~\ref{alg:zoom} behaves like binary search, using directional derivatives as needed to determine which half of the search window should remain active. If the bisection point has the lowest function value seen so far, then $h'$ is computed to determine the recurrence interval; otherwise only an evaluation of $h$ is necessary. Regardless of this distinction, however, every iteration of Algorithm~\ref{alg:zoom} decreases the size of the active search window by a factor of 2, terminating when a trial point satisfies the strong Wolfe curvature condition. Let $\alpha_L$ and $\alpha_{\hat R}$ denote the left and right endpoints of the search window at the beginning of this phase of Algorithm~\ref{alg:zoom} ($\alpha_{\hat R}$ may be equal to $\alpha_R$ or smaller, depending on $q$). We know that both of $\alpha_L$ and $\alpha_{\hat R}$ satisfy the Armijo condition, so $\alpha_{\hat R} - \alpha_L \leq \frac{-d^T\nabla f(x)}{m\norm{d}^2}$. Let $r$ be the number of iterations in this phase of Algorithm~\ref{alg:zoom}, and $A$ be the difference between the largest and smallest step sizes that satisfy the curvature condition. Then $r < \max \left( 0, 1 + \log_2 \left(\frac{\alpha_{\hat R} - \alpha_L}{A} \right) \right)$. Finally, we can bound $A$ using Lipschitz gradient: $A \geq \frac{-2c_2d^T\nabla f(x)}{L\norm{d}^2}$. Putting these together, we have $r < 1 + \log_2 \left(\frac{L}{2mc_2} \right)$. In the worst case, each of these $r$ iterations may require evaluation of both the function and directional derivative.

The worst case total number of directional derivative evaluations is upper bounded by $p + r < \max(2 + \log_{1/\beta}\left(-\frac{d^T\nabla f(x)}{m\alpha_1 \norm{d}^2} \right), 1) + 1 + \log_2 \left(\frac{L}{2mc_2} \right)$. The worst case total number of function evaluations is upper bounded by $p + q + r < \max(2 + \log_{1/\beta}\left(-\frac{d^T\nabla f(x)}{m\alpha_1 \norm{d}^2} \right), 1) + \max \left(0, \log_2\left(\frac{-\alpha_1L\norm{d}^2}{d^T\nabla f(x)}\right)\right) + 1 + \log_2 \left(\frac{L}{2mc_2} \right) = \max(2 + \log_{1/\beta}\left(-\frac{d^T\nabla f(x)}{m\alpha_1 \norm{d}^2} \right), 1) + \log_2\left(\max \left(\frac{L}{mc_2}, \frac{-\alpha_1 L^2\norm{d}^2}{mc_2d^T\nabla f(x)} \right) \right)$.
\end{proof}

Note that the directional derivative is always computed at the same location as a function evaluation; if the directional derivative is estimated by finite differencing, this requires only one additional function evaluation. Of course, the directional derivative can also be computed using the full gradient, if that is available and efficient to evaluate.

\section{Empirical results}
\label{sec:experiments}

We evaluate the empirical performance of approximately exact line search and baseline approaches in three settings: deterministic gradient descent for logistic regression, minibatch (stochastic) gradient descent for logistic regression, and derivative-free optimization (DFO). The objective function for logistic regression is strongly convex (due to a regularizer) with Lipschitz gradient, but these assumptions do not hold on the DFO benchmark problems \cite{more2009benchmarking}. 

All experiments are run on a 4 socket Intel Xeon CPU E7-8870 machine with 18 cores per socket and 1TB of DRAM. We run experiments in parallel, ensuring that at most one trial is run concurrently on each core.

\subsection{Logistic regression}
Our logistic regression experiments use the following objective function (with $\lambda = \frac{1}{N}$): 
\begin{align}
f(x) = \frac{\lambda}{2} x^Tx + \frac{1}{N}\sum_{i=1}^N \log(1 + \exp(-y_i(x^Tz_i)))\;,
\end{align}
where $x$ is a vector of parameters to be optimized, $N$ is the number of data points, and each data point has covariates $z_i$ and binary label $y_i$. This objective is commonly used in practice, has $L$-Lipschitz gradient with $L \leq \lambda + 2\lambdamax(\frac{1}{N}\sum_{i=1}^N z_i z_i^T)$ (where $\lambdamax(\cdot)$ denotes the maximum eigenvalue of a positive semidefinite matrix), and is $m$-strongly convex with $m \geq \lambda$. 

We define a benchmark set of logistic regression objectives comprised of the three datasets summarized in Table~\ref{table:datasets}, where we append $1$ to each covariate vector to allow for a bias in the model (this extra dimension is not included in Table~\ref{table:datasets}). For each problem, we use the origin as a standard initial point $x_0$. We estimate a Barzilai-Borwein \cite{bb} step size $\tbb$ there (using the full dataset), and include each experiment using various initial step sizes $T_0$: $0.01\tbb$, $0.1\tbb$, $\tbb$, $10\tbb$, and $100\tbb$. We include this range of initial step sizes to highlight the insensitivity of AELS to this parameter. Our code is available at \url{https://github.com/modestyachts/AELS}.

\begin{table}[h!]
  \begin{center}
    \caption{Datasets used for logistic regression experiments.}
    \label{table:datasets}
    \begin{tabular}{l|r|r}
	Dataset & Data points & Variables \\
      \hline
      adult (LIBSVM \cite{libsvm} version of UCI \cite{UCI} dataset) & 32561 & 123 \\
      quantum (from KDD Cup 2004 \cite{Caruana:2004:KRA:1046456.1046470}) & 50000 & 78\\
      protein (from KDD Cup 2004 \cite{Caruana:2004:KRA:1046456.1046470}) & 145751 & 74\\
    \end{tabular}
  \end{center}
\end{table}

In our logistic regression experiments, we compare the following line searches:
\begin{itemize}
\item Approximately exact: Algorithm~\ref{alg:approximately exact line search with warm start}.
\item Adaptive backtracking: Follows the same ``warm start'' initial step size as AELS and decreases it geometrically until the Armijo condition is satisfied.
\item Traditional backtracking: At every step, starts with the same initial step size $T_0$ and decreases it geometrically until the Armijo condition is satisfied.
\item Wolfe: Line search procedure for the strong Wolfe conditions from Nocedal and Wright \cite{nocedal2006line}, based on the implementation in \\{\tt scipy.optimize} \cite{scipy}. This line search uses directional derivative evaluations (implemented using the gradient) as well as function evaluations.
\item Constant: Every step uses the fixed step size $T_0$.
\item Inverse: In step $i=1,2,3,\dots$, the step size is $T_0/i$.
\end{itemize}

We plot results using performance profiles, introduced by Mor\'{e} and Wild \cite{more2009benchmarking} to compare computational efficiency in terms of function evaluations for the DFO setting. Performance profiles visualize the fraction of benchmark problems (on the vertical axis) solved within various multiples of the computational cost incurred by the fastest algorithm to achieve convergence on each problem (performance ratio, on the horizontal axis). For instance, if the curve for a given algorithm starts at 0.7 at performance ratio 1 and reaches 1 at performance ratio 6, it means that this was the computationally cheapest algorithm for 70\% of the test problems, but that on at least one test problem it incurred 6 times the cost of the cheapest algorithm for that problem. For our logistic regression experiments, since all algorithms use both function and gradient evaluations, we use computation time rather than function evaluations as our notion of computational cost.

\subsubsection{Deterministic gradient descent}
We begin our logistic regression experiments by executing each algorithm as described, using the full dataset to compute the objective function and its gradient exactly wherever required. In this setting, the upper bounds on function and gradient evaluations from Section \ref{sec:theory} apply. Our experiments in this section, shown in Figure~\ref{fig:logistic_gd}, confirm they are indeed predictive of empirical performance both between the algorithms we consider and across the range of initial step sizes $T_0$ we test for each objective function. 

\begin{figure}[h]
\centering
\includegraphics[width=0.95\textwidth]{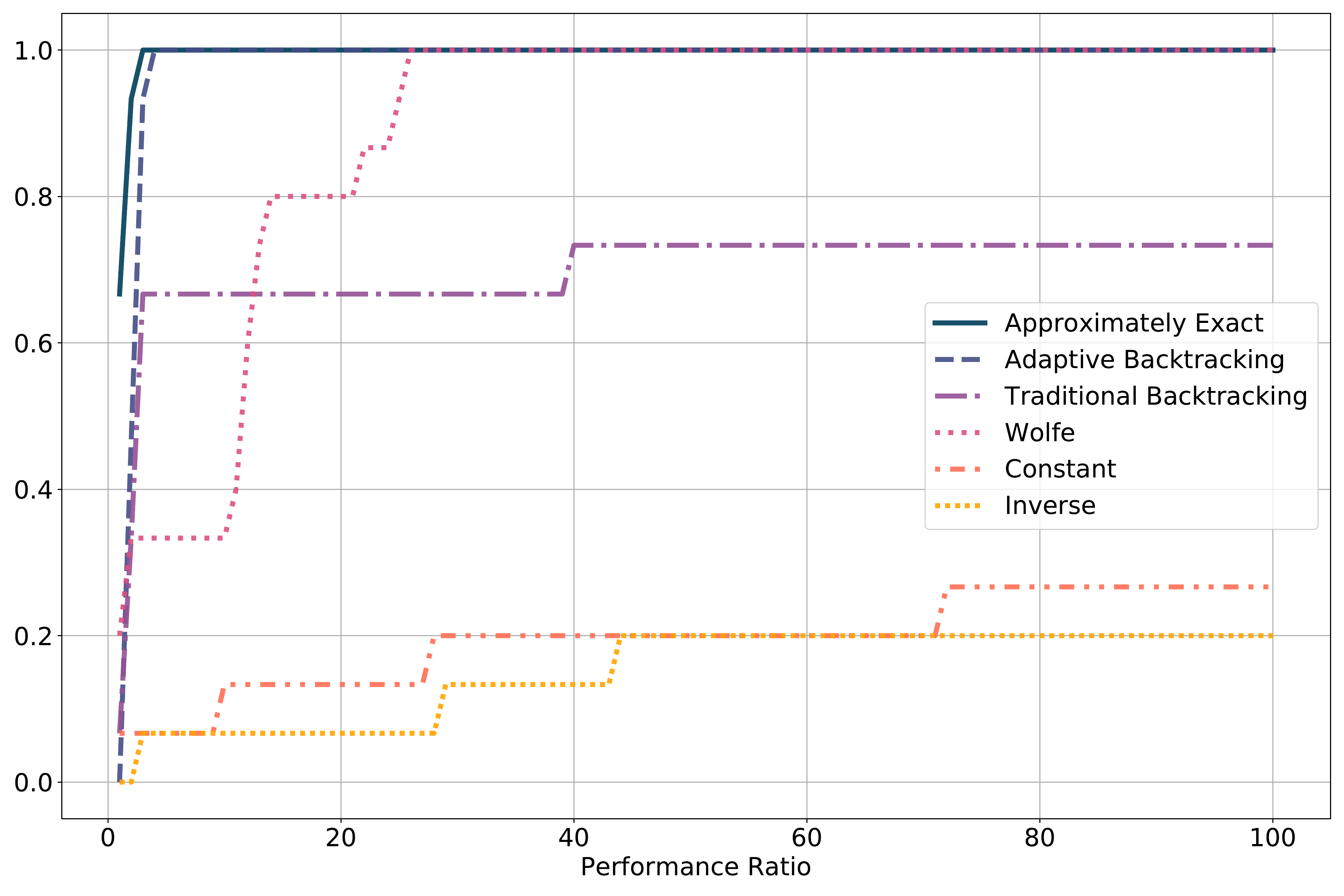}
\caption{Performance profile for gradient descent (to a relative error of $10^{-4}$) on our logistic regression benchmark problems.}
\label{fig:logistic_gd}
\end{figure}

Traditional backtracking line search performs well with a sufficiently large $T_0$, but slows down dramatically if $T_0$ is too small. Adaptive backtracking largely fixes this problem, at the cost of some overhead in the number of steps needed to reach a good step size. AELS achieves performance similar to that of traditional backtracking with a good $T_0$, without any notable dependence on $T_0$. As expected, the fixed step size schedules Constant and Inverse are sensitive to the choice of $T_0$ and for most initial step sizes they are not competitive with any of the line searches. Wolfe line search is not sensitive to $T_0$, but is slowed down by expensive gradient evaluations required during each line search.

\subsubsection{Minibatch stochastic gradient descent}
Although our analysis does not extend to the stochastic function setting, in practice with large datasets it is often desirable to work with only a small subset (minibatch) of the data at a time. In this section, we repeat the above experiment except that at each iteration we randomly sample a minibatch from the dataset and replace all required objective function and gradient values with their approximations using only the current minibatch. 

As we saw in Section \ref{sec:theory}, our proposed algorithm makes the tradeoff of slightly more expensive steps in exchange for requiring fewer steps to converge, resulting in a net speedup. In the minibatch setting, our algorithm still requires more computation per step than simpler line searches, but since we are optimizing an approximate objective in each step it is \emph{a priori} unclear whether this investment should enable any decrease in the number of steps to convergence. 

In particular, there are two (not necessarily disjoint) regimes of stochastic optimization in which we might expect a more expensive line search subroutine to run slower than a simpler subroutine: when the minibatch size is too small and when the iterates approach the optimum. 

In the first case, the approximate objective we optimize in each step can be far from the true objective, so putting more computational effort into choosing each step size might be counterproductive. We explore this issue in Figure~\ref{fig:logistic_sgd_batchsize} by repeating our logistic regression experiment with decreasing minibatch sizes, and find that AELS remains competitive across a wide range of minibatch sizes, including the extreme case of single-example minibatches. Note that it is possible to construct adversarial datasets for which single-example (or very small minibatch) AELS would oscillate rather than converge; if such behavior is observed in practice it can be remedied by increasing the minibatch size, adjusting the search directions to incorporate history from prior minibatches, or adding a trust region. 

We also find that, although Wolfe line search outperforms both fixed step size schedules and traditional backtracking line search in the full batch regime, its relative performance degrades markedly as the minibatch size decreases and the gradient noise correspondingly increases. We hypothesize that this behavior results from the difference in gradient dependence of the different line searches: Wolfe line search requires multiple gradient evaluations inside the line search, backtracking line searches use only the gradient at the initial point of the line search, and AELS and fixed step size schedules require zero gradient evaluations. During minibatch SGD the gradient noise is typically larger than the function noise, so it is not surprising that line searches based on function evaluations are more robust in this setting.

The second scenario (approaching the minimum) is similar, except that for functions with wide, flat minima we run the risk of taking a large counterproductive step even with a minibatch size that might be large enough to make progress farther from the minimum. We explore this issue in Figure~\ref{fig:logistic_sgd_epsilon}, in which we hold the minibatch size constant (at 100) and vary the relative accuracy required for convergence, forcing the iterates to approach the minimizer. We find that AELS remains fastest until the required relative error drops to $10^{-6}$, at which point it performs similarly to adaptive backtracking line search. As in the case of decreasing minibatch size, as we approach the minimizer we find that AELS and adaptive backtracking line search remain competitive, whereas Wolfe line search ultimately becomes slower than even the fixed step size schedules. Although performance may vary across objective functions, this experiment gives a sense of how close each method can get to the minimum before the line searches are no longer helpful.

\begin{figure}[H]
\centering
\begin{subfigure}{0.49\textwidth}
        \includegraphics[width=\textwidth]{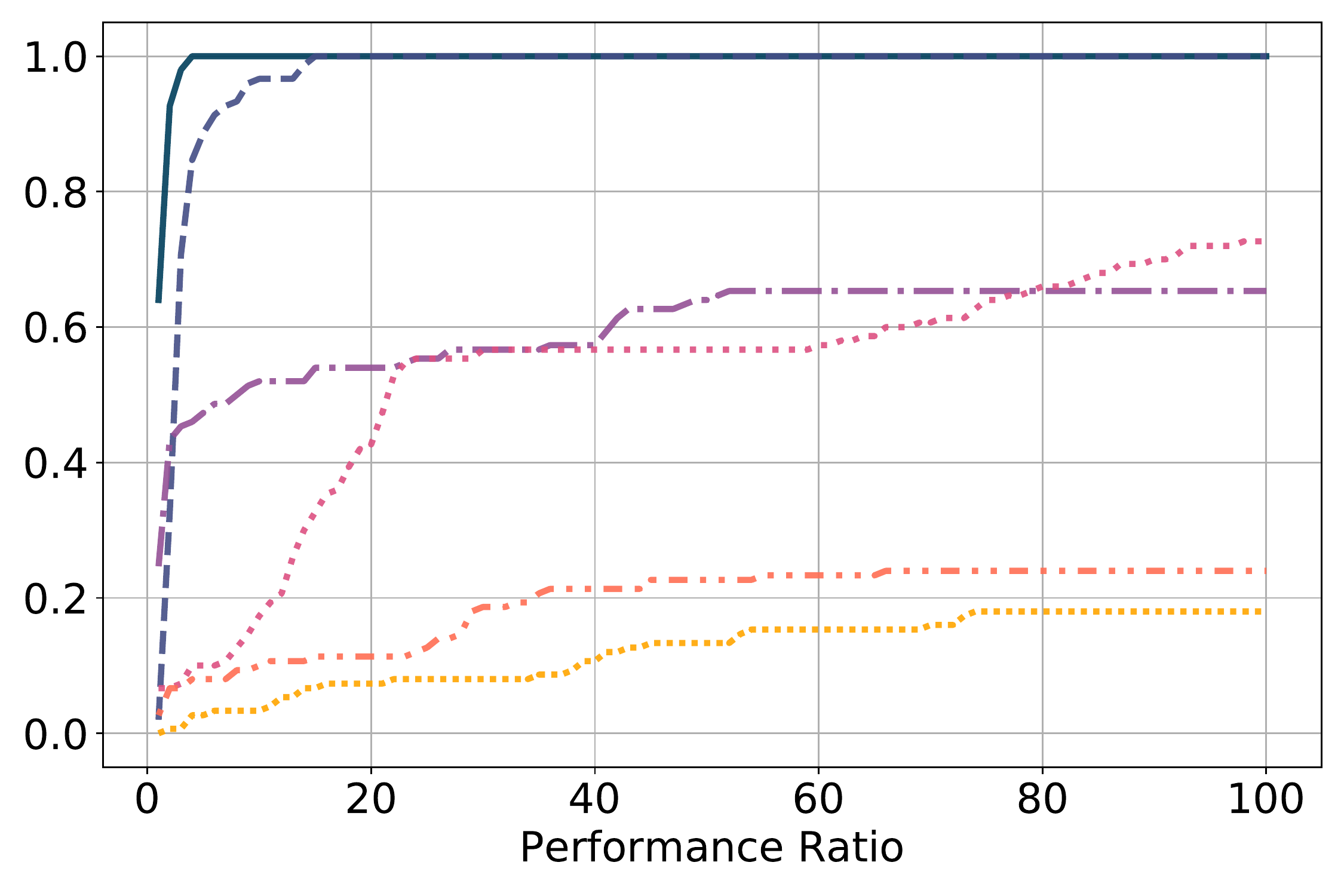}
        \caption{Minibatch size: 1600}
        \label{fig:1600}
    \end{subfigure}
    \hfill
    \begin{subfigure}{0.49\textwidth}
        \includegraphics[width=\textwidth]{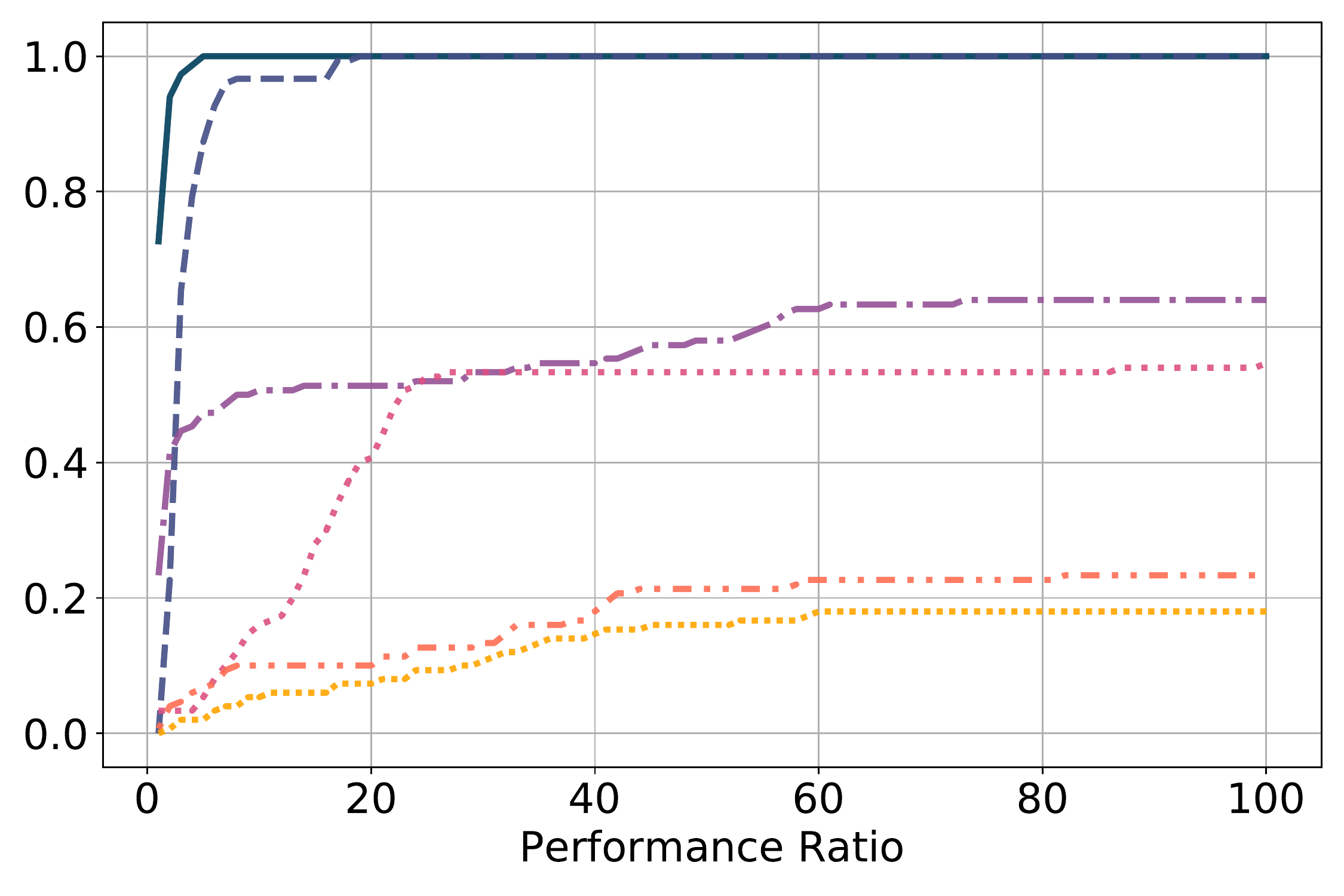}
        \caption{Minibatch size: 400}
        \label{fig:400}
    \end{subfigure}
    
    \begin{subfigure}{0.49\textwidth}
        \includegraphics[width=\textwidth]{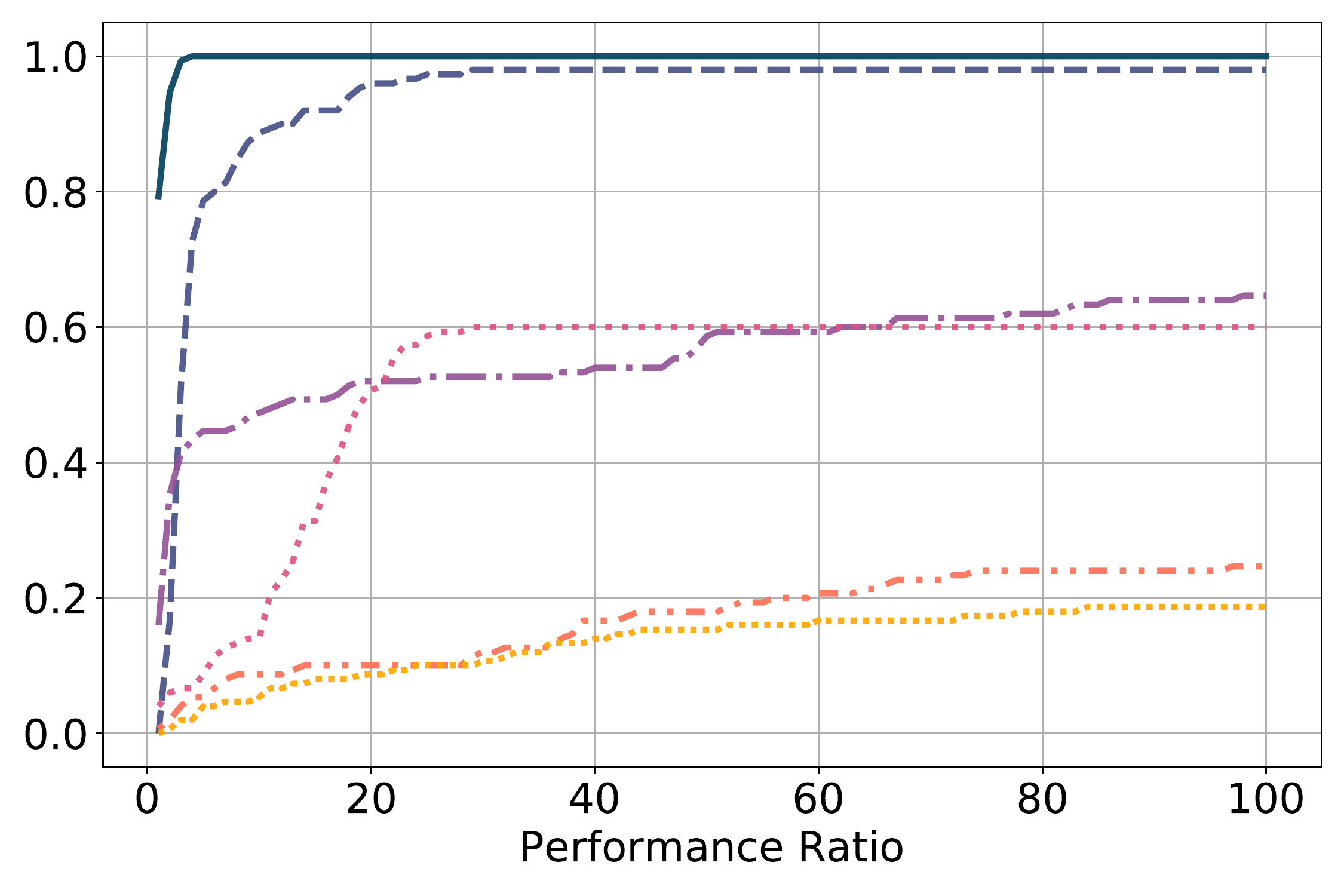}
        \caption{Minibatch size: 100}
        \label{fig:100}
    \end{subfigure}
    \hfill
    \begin{subfigure}{0.49\textwidth}
        \includegraphics[width=\textwidth]{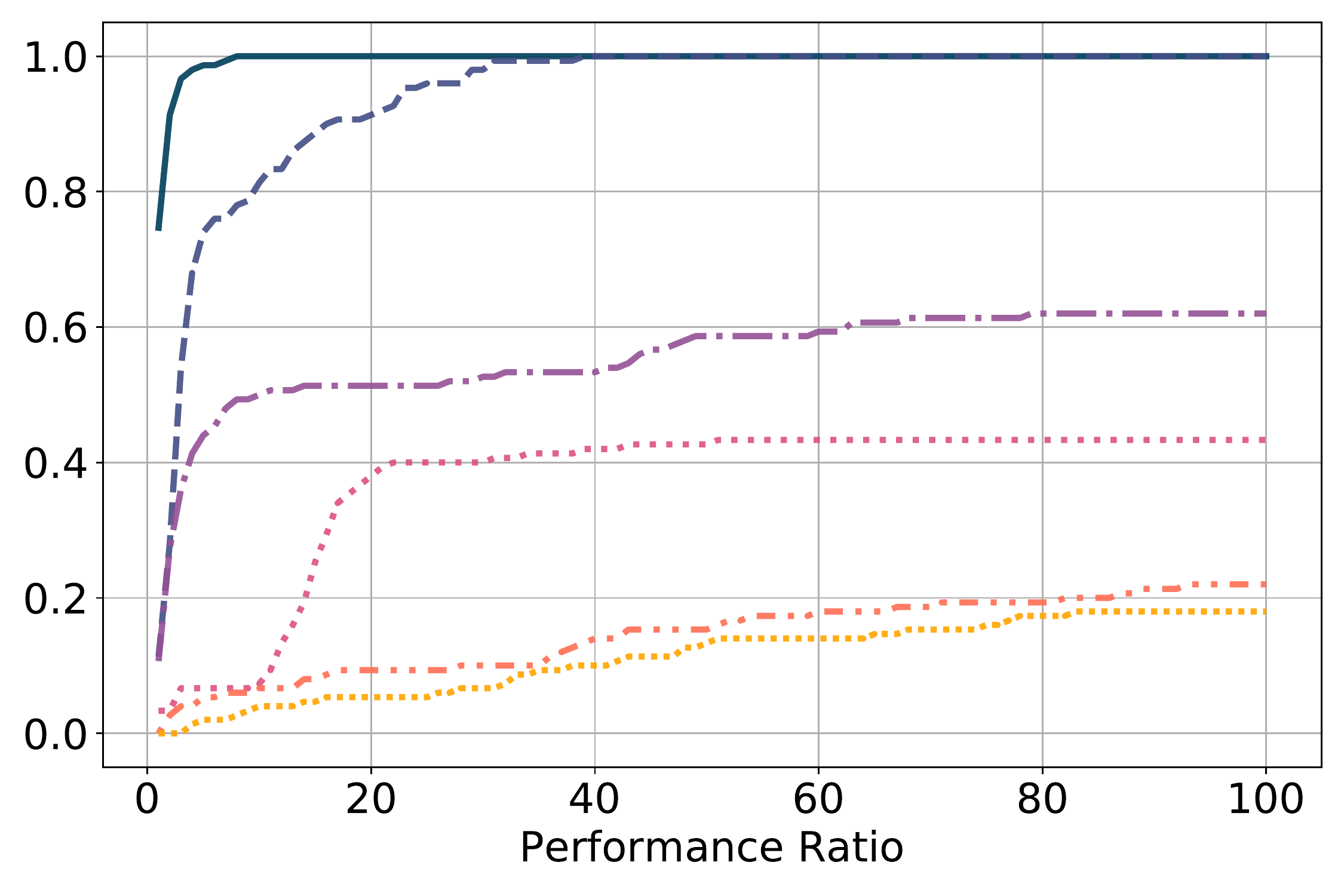}
        \caption{Minibatch size: 25}
        \label{fig:25}
    \end{subfigure}
    
    \begin{subfigure}{0.49\textwidth}
        \includegraphics[width=\textwidth]{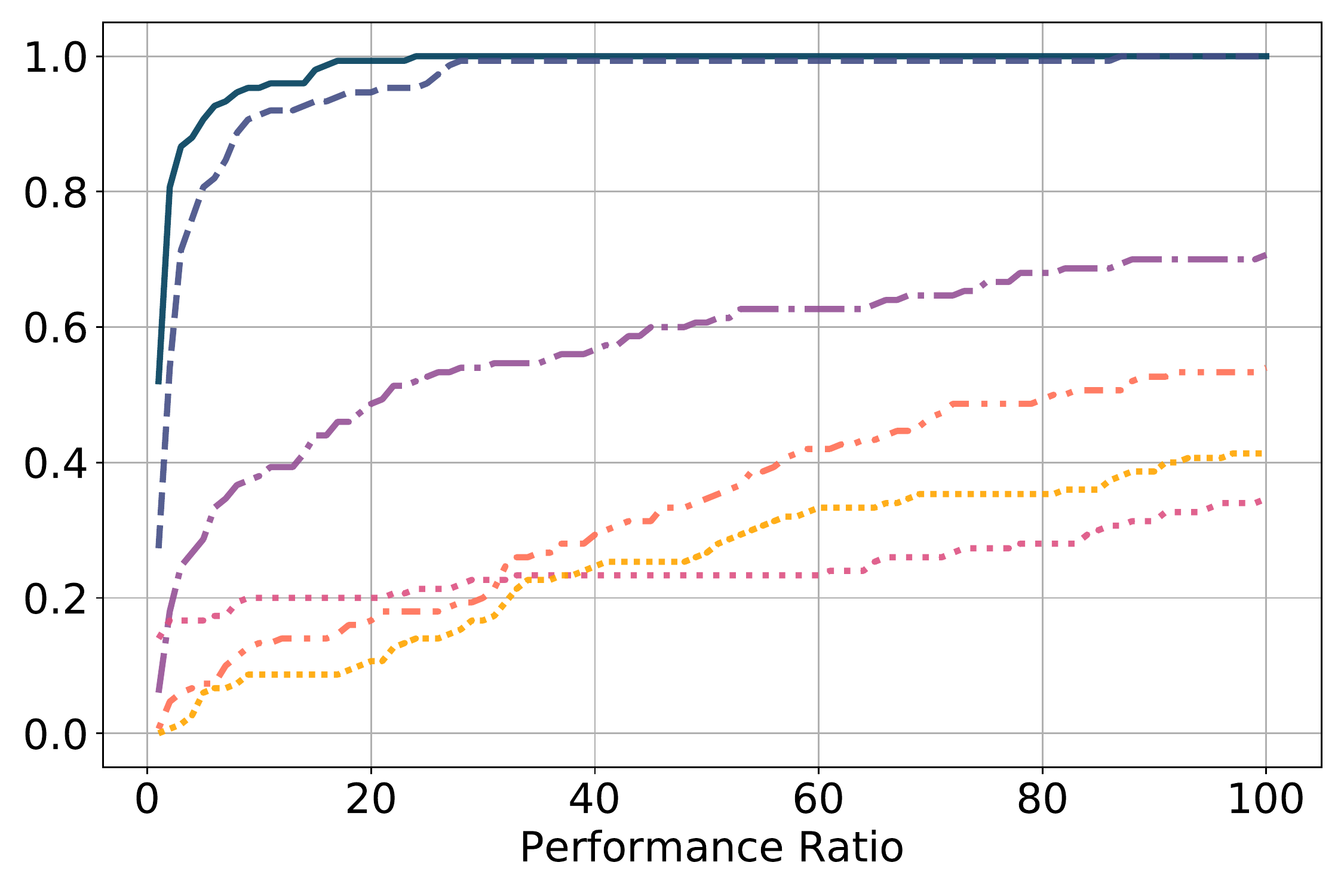}
        \caption{Minibatch size: 6}
        \label{fig:6}
    \end{subfigure}
    \hfill
    \begin{subfigure}{0.49\textwidth}
        \includegraphics[width=\textwidth]{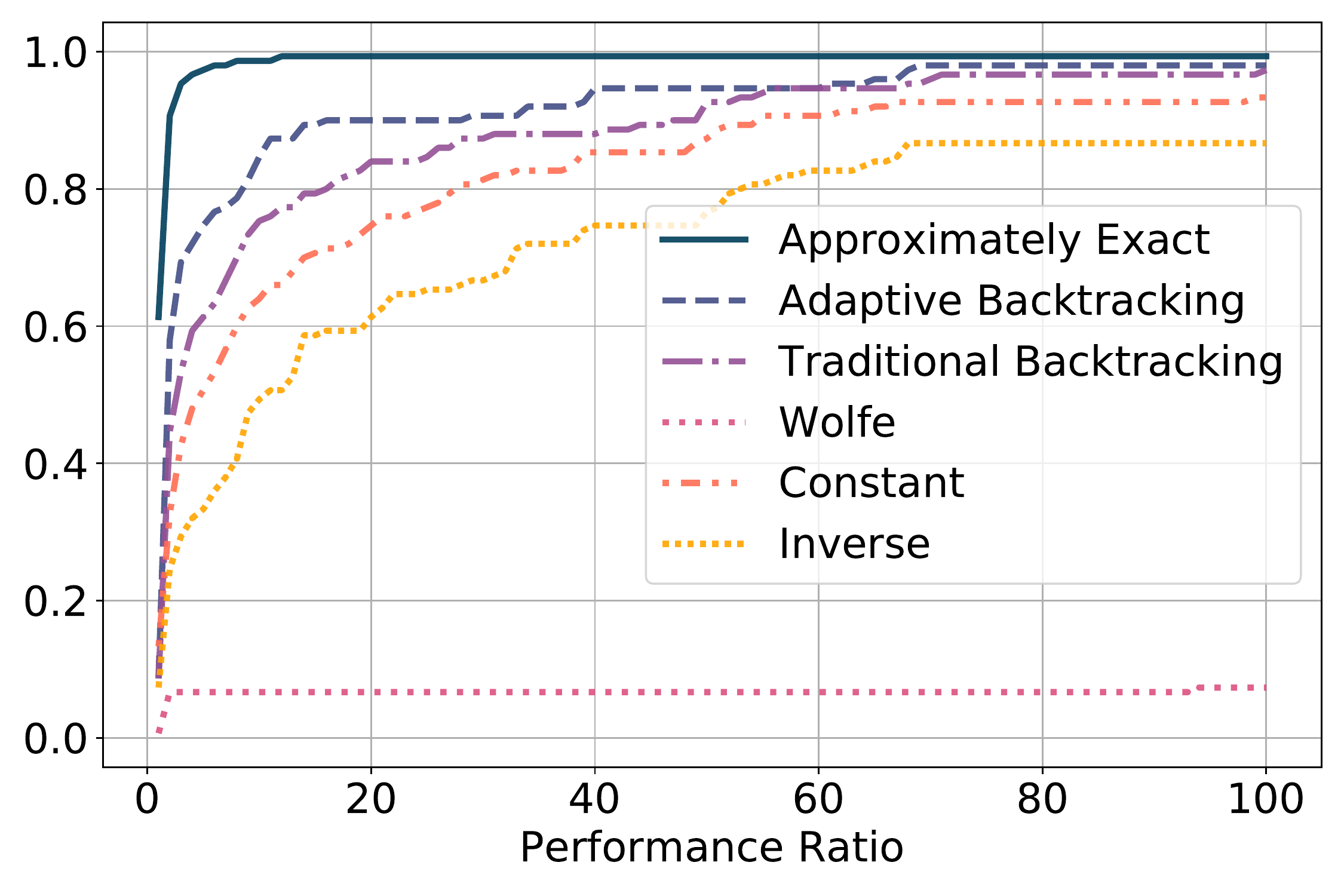}
        \caption{Minibatch size: 1}
        \label{fig:1}
    \end{subfigure}
\caption{Performance profiles for stochastic gradient descent (to a relative error of $10^{-4}$) on our logistic regression benchmark problems, repeated with varying minibatch sizes and over 10 random seeds (using the same set of seeds for each algorithm). The performance improvement of AELS in the full batch regime (Figure~\ref{fig:logistic_gd}) transfers to the minibatch regime across all minibatch sizes tested.}
\label{fig:logistic_sgd_batchsize}
\end{figure}

These experiments suggest that even our relatively expensive line search method can improve performance over a wide range of minibatch sizes and required accuracies, likely including many practical scenarios. Additionally, the line search method we propose could be easily combined with an adaptive batch size or other variance reduction or trust region method, as needed for specific problems.

\begin{figure}[t]
\centering
\begin{subfigure}[t]{0.49\textwidth}
        \includegraphics[width=\textwidth]{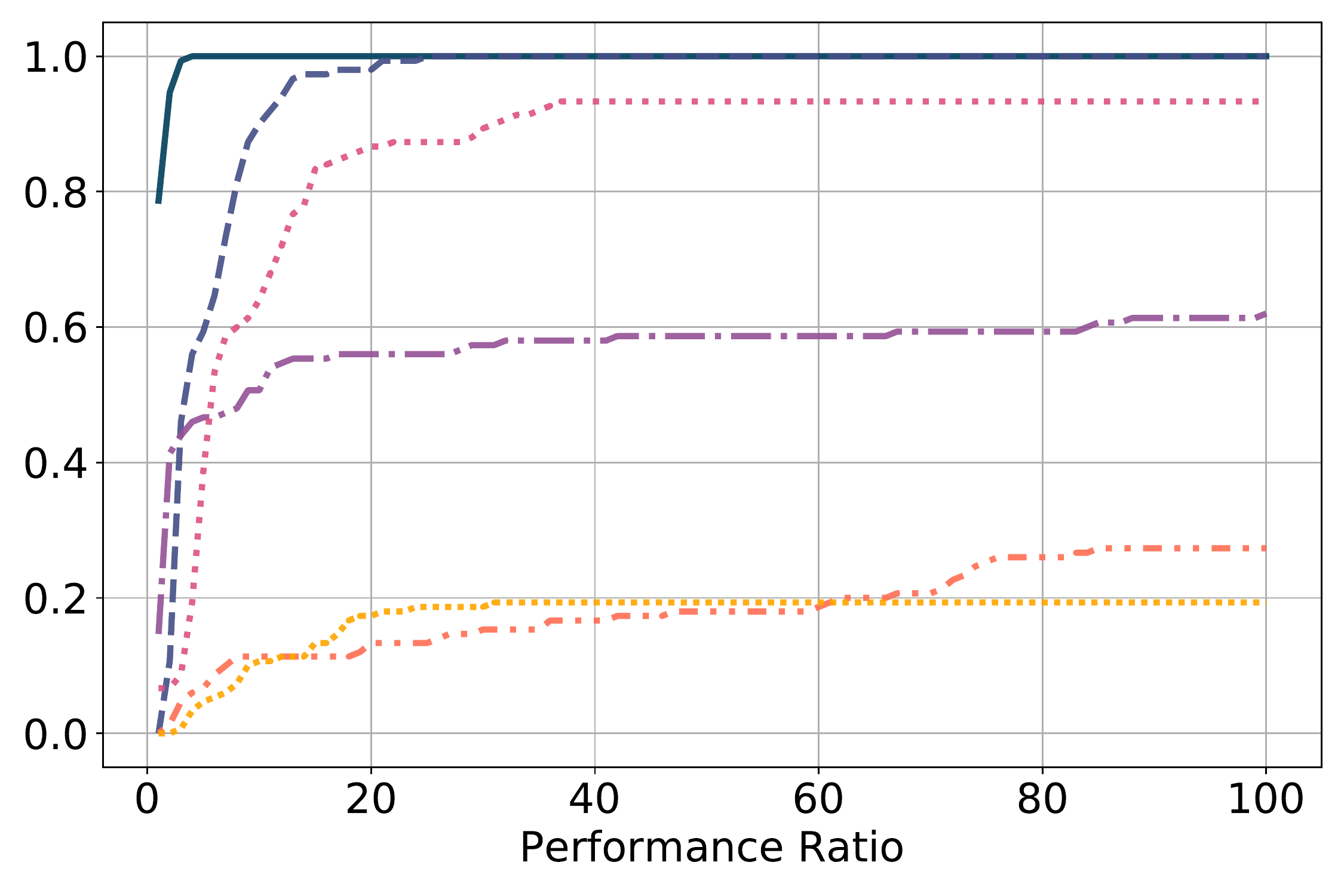}
        \caption{Relative error: $10^{-3}$}
        \label{fig:1e-3}
        \end{subfigure}
        \hfill
        \begin{subfigure}[t]{0.49\textwidth}
        \includegraphics[width=\textwidth]{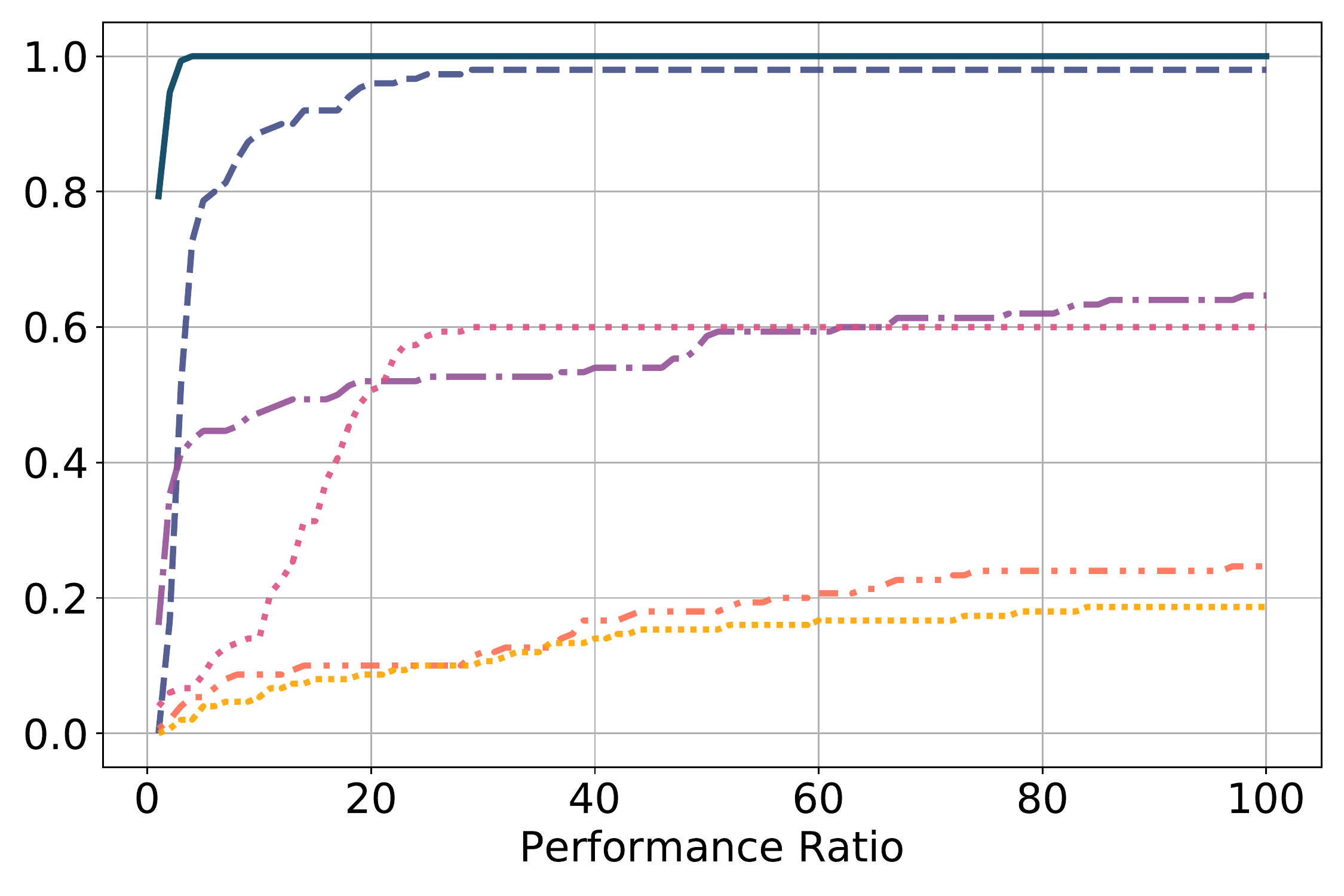}
        \caption{Relative error: $10^{-4}$}
        \label{fig:1e-4}
        \end{subfigure}
        
        \begin{subfigure}[t]{0.49\textwidth}
        \includegraphics[width=\textwidth]{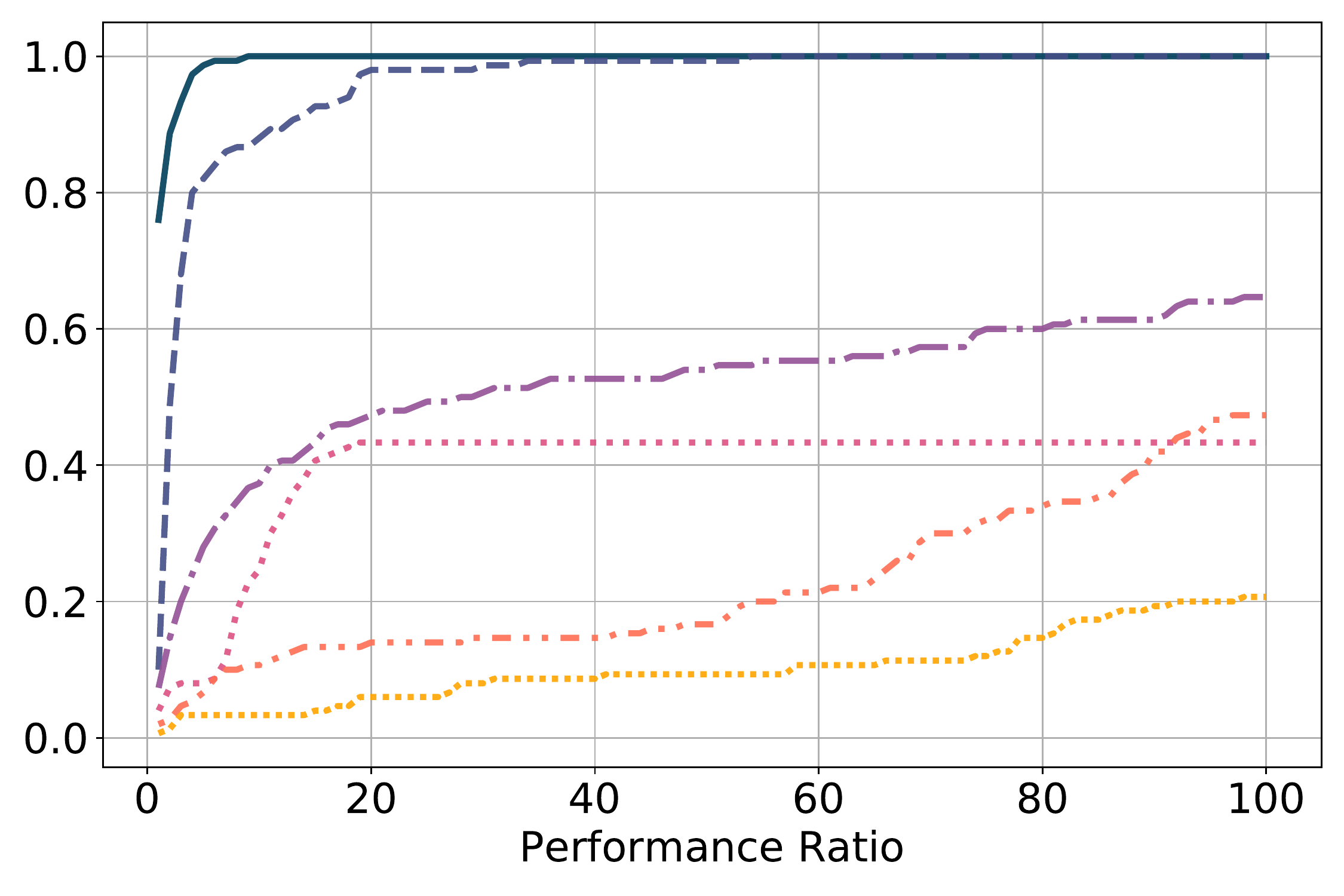}
        \caption{Relative error: $10^{-5}$}
        \label{fig:1e-5}
        \end{subfigure}
        \hfill
        \begin{subfigure}[t]{0.49\textwidth}
        \includegraphics[width=\textwidth]{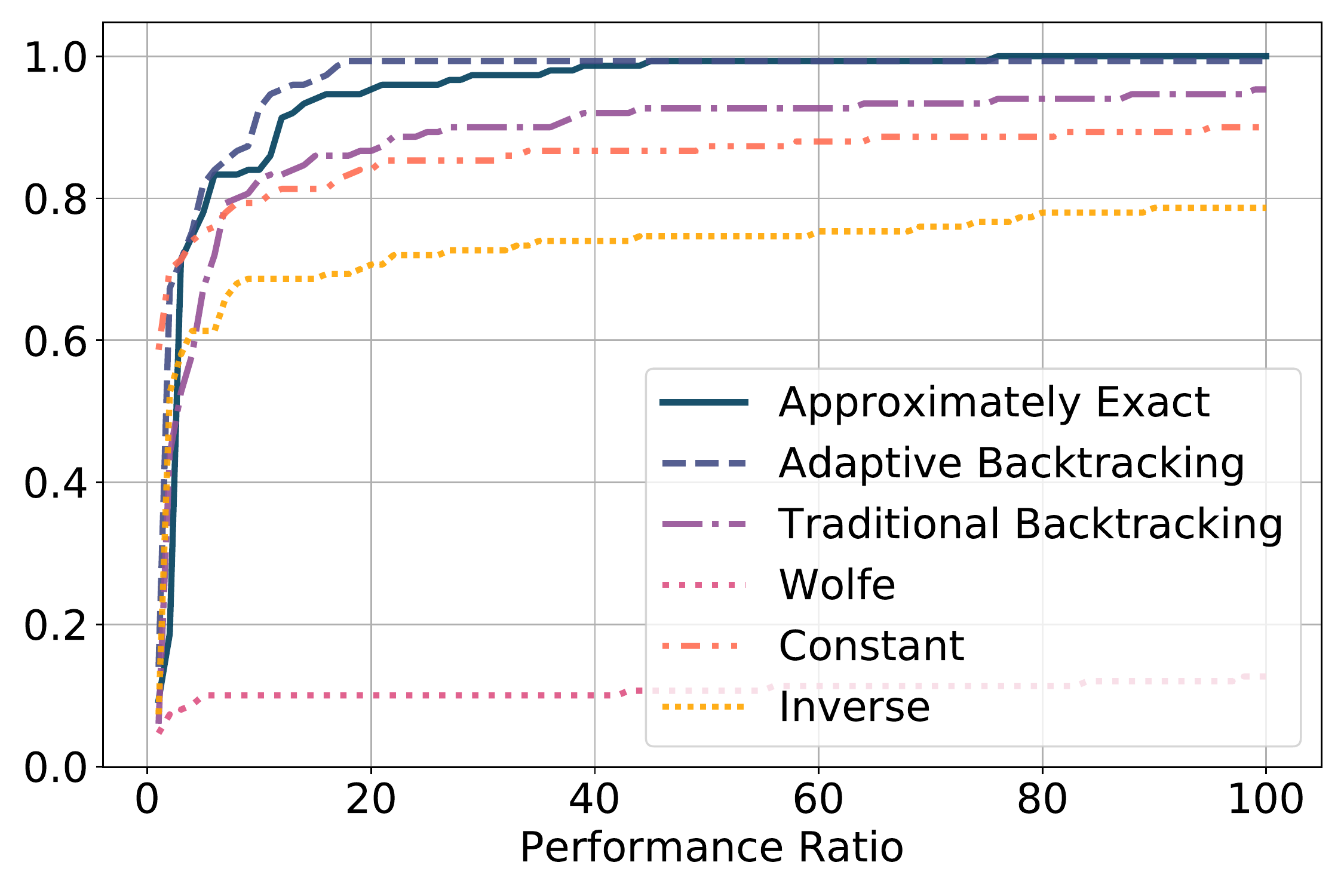}
        \caption{Relative error: $10^{-6}$}
        \label{fig:1e-6}
        \end{subfigure}
\caption{Performance profiles for stochastic gradient descent (with minibatches of size $100$) on our logistic regression benchmark problems, repeated with varying required accuracy and over 10 random seeds (using the same set of seeds for each algorithm). The performance improvement of AELS in the minibatch regime (Figure~\ref{fig:logistic_sgd_batchsize}) transfers to the various accuracy regimes tested, with performance degrading to roughly the same as adaptive backtracking in the high accuracy regime (relative error $10^{-6}$).}
\label{fig:logistic_sgd_epsilon}
\end{figure}

\subsection{Derivative-free optimization}
Our derivative-free optimization (DFO) experiments use the 53 benchmark problems proposed by Mor{\'{e}} and Wild \cite{more2009benchmarking}, all with an initial step size $T_0=1$. These problems are low to medium dimensional (no more than 12 dimensions), smooth (although accessed only through function evaluations), generally nonconvex (although often have unique minima), and include standard starting points sometimes near and sometimes far relative to a minimizer. We use BFGS \cite{bfgs} descent directions for all of our line search methods, because many of the DFO objectives are poorly conditioned and pure steepest descent is needlessly slow (\ie outperformed by Nelder-Mead). Note that the BFGS update step is only guaranteed to be well-defined if the step size satisfies the Wolfe conditions. We account for this by explicitly checking if the estimated BFGS direction is a descent direction; otherwise, we use the negative gradient instead and reset our BFGS history.

Our finite differencing uses the same sampling radius ($\sigma \approx 1.5 \times 10^{-8}$) as recommended by Nocedal and Wright \cite{nocedal2006line} and used in {\tt scipy.optimize} \cite{scipy}. We experimented with a heuristic for adapting $\sigma$, but found a fixed $\sigma$ worked better on these problems. We also experimented with central finite differencing, but found forward finite differencing to be a bit faster, offering a favorable tradeoff of fewer function evaluations for minimal lost accuracy. We use the same forward finite differencing approximation in place of the gradient in both gradient descent and BFGS \cite{bfgs}. For the directional derivatives required by the Wolfe line search, we use a two-point forward finite differencing estimate that is more efficient than approximating the full gradient (currently {\tt scipy.optimize} estimates the full gradient, but we have created a pull request to enable this faster option).

We compare the following methods:
\begin{itemize}
\item BFGS with approximately exact line search (Algorithm~\ref{alg:approximately exact line search with warm start}).
\item BFGS with adaptive backtracking line search.
\item BFGS with traditional backtracking line search.
\item BFGS with a Wolfe line search \cite{wolfe1969convergence}; our implementation is based on {\tt scipy.optimize} \cite{scipy} but uses a cheaper two-point directional derivative approximation that avoids estimating the full gradient.
\item Nelder-Mead \cite{nelder1965simplex}, as implemented in {\tt scipy.optimize} \cite{scipy}.
\end{itemize}

\begin{figure}[h]
\centering
\includegraphics[width=0.9\textwidth]{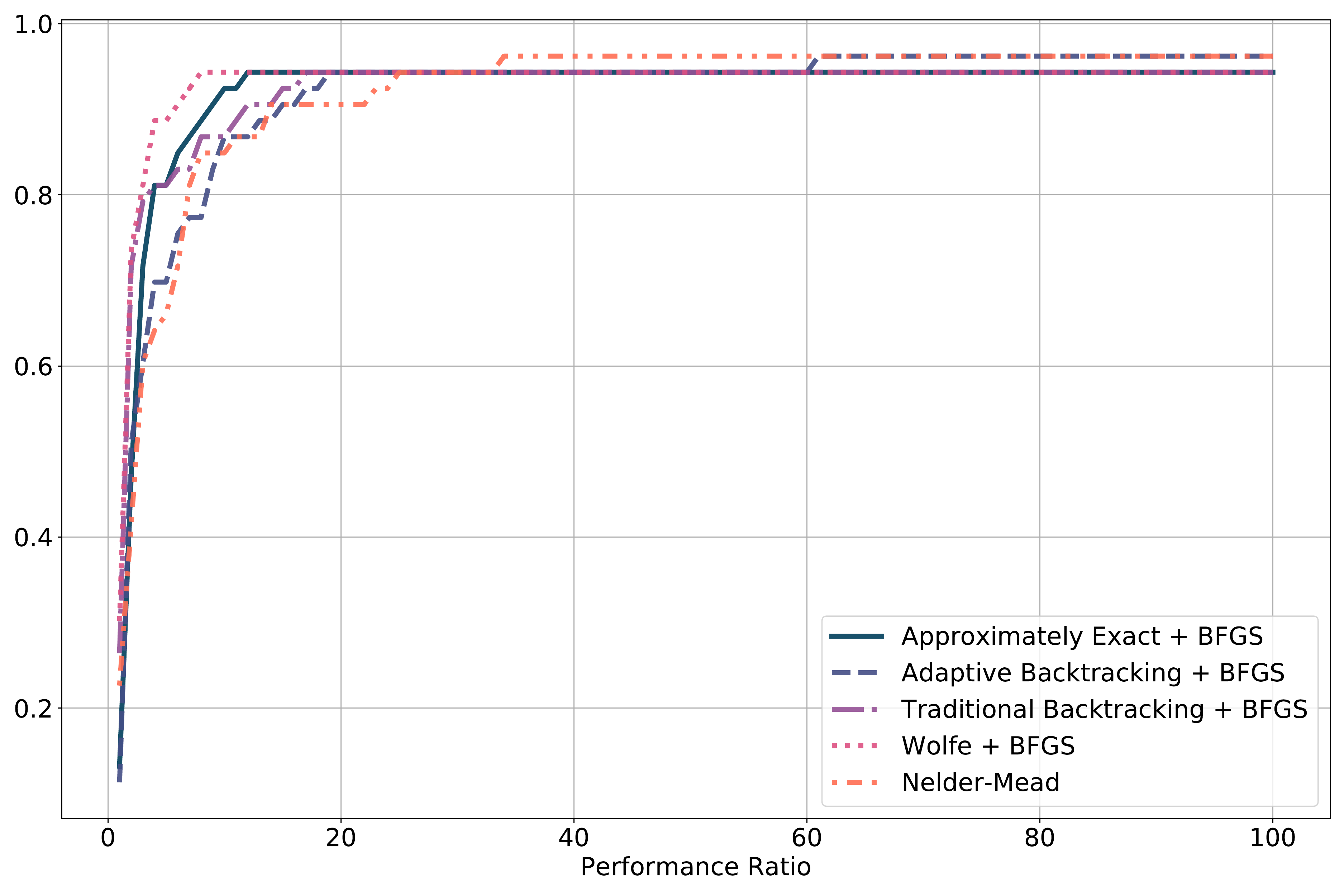}
\caption{Performance profile on the DFO benchmark of Mor{\'{e}} and Wild \cite{more2009benchmarking}.}
\label{fig:dfo}
\end{figure}

Figure~\ref{fig:dfo} shows a performance profile comparing these methods. We find that even on this challenging derivative-free benchmark with finite differencing gradient estimates, all of these line search methods with BFGS search directions are competitive with Nelder-Mead, an explicitly derivative-free algorithm. Of these, Wolfe line search is slightly faster than AELS, which is in turn slightly faster than the Armijo backtracking methods. This similar performance between Wolfe and AELS is in line with the upper bounds on iterations and function evaluations in Section~\ref{sec:theory}.

\section{Conclusions}
\label{sec:conclusions}

We introduce a simple line search method, approximately exact line search (AELS, Algorithm~\ref{alg:approximately exact line search with warm start}), that uses only function evaluations to find a step size within a constant fraction of the exact line search minimizer. Our algorithm can be thought of as an adaptive version of golden section search \cite{avriel1968golden}. We prove linear convergence of our method applied to gradient descent as well as derivative-free optimization with either a finite differencing gradient approximation or a random search direction (in which case we prove linear expected convergence). We also bound the number of function evaluations per iteration as logarithmic in properties of the objective function and gradient approximation. We compare these bounds to those for Armijo line searches such as the popular backtracking method, and find that our method is more robust to improper choices of the initial step size $T_0$ as well as better suited to the derivative-free setting, in which Armijo methods depend more strongly on the quality of the gradient approximation. We also prove parallel bounds on the number of iterations and function evaluations for a strong Wolfe line search, a special case of that presented in Nocedal and Wright \cite{nocedal2006line}, and find them to be quite similar to our bounds for AELS, although Wolfe line search makes use of directional derivative evaluations in addition to function evaluations.

We illustrate these worst-case bounds using gradient descent on a strongly convex logistic regression benchmark, on which we find that our method converges faster than Wolfe line search, Armijo backtracking methods, and simple fixed step size schedules across a range of initial step sizes $T_0$. We also stress-test our approach in the stochastic setting by repeating our logistic regression benchmark trials with decreasing minibatch size and decreasing relative error at convergence. Although line search is commonly believed to be useless or counterproductive in the stochastic setting, in which our analysis does not directly apply, we find that AELS remains performant on logistic regression even with single-example minibatches or when optimizing to a relative error less than $10^{-5}$. These regimes are of great practical interest, and our experiments suggest that AELS may be effective well beyond the settings we analyzed. We believe this robustness to stochasticity is largely due to the derivative-free nature of AELS, since gradients often experience more substantial noise than function evaluations.

Finally, we present empirical results of our method using BFGS search directions on a DFO benchmark (including nonconvex but smooth objectives), and show faster or at least competitive convergence compared to popular baselines including Nelder-Mead and BFGS with a Wolfe line search. These experiments suggest that, although our convergence analysis focused on strongly convex functions, with proper implementation and search directions (as provided in our code) our method extends gracefully to nonconvex objectives. This result is encouraging but unsurprising, since our line search method itself relies only on unimodality rather than convexity.

In total, our investigation indicates that line search, particularly AELS, improves performance in many regimes, including ones in which line search is typically avoided. We hope that our work spurs further theoretical research in characterizing the performance of approximately exact line search and other line searches in the stochastic and nonconvex settings, as well as further empirical research exploring the range of practical optimization problems for which these line searches improve performance.

\backmatter

\bmhead{Acknowledgments}

We appreciate helpful feedback from Ludwig Schmidt.
This research is generously supported in part by ONR awards N00014-20-1-2497 and N00014-18-1-2833, NSF CPS award 1931853, and the DARPA Assured Autonomy program (FA8750- 18-C-0101). 
SFK is also supported by NSF GRFP.

\bibliography{references}


\begin{thebibliography}{40}
\ifx \bisbn   \undefined \def \bisbn  #1{ISBN #1}\fi
\ifx \binits  \undefined \def \binits#1{#1}\fi
\ifx \bauthor  \undefined \def \bauthor#1{#1}\fi
\ifx \batitle  \undefined \def \batitle#1{#1}\fi
\ifx \bjtitle  \undefined \def \bjtitle#1{#1}\fi
\ifx \bvolume  \undefined \def \bvolume#1{\textbf{#1}}\fi
\ifx \byear  \undefined \def \byear#1{#1}\fi
\ifx \bissue  \undefined \def \bissue#1{#1}\fi
\ifx \bfpage  \undefined \def \bfpage#1{#1}\fi
\ifx \blpage  \undefined \def \blpage #1{#1}\fi
\ifx \burl  \undefined \def \burl#1{\textsf{#1}}\fi
\ifx \doiurl  \undefined \def \doiurl#1{\url{https://doi.org/#1}}\fi
\ifx \betal  \undefined \def \betal{\textit{et al.}}\fi
\ifx \binstitute  \undefined \def \binstitute#1{#1}\fi
\ifx \binstitutionaled  \undefined \def \binstitutionaled#1{#1}\fi
\ifx \bctitle  \undefined \def \bctitle#1{#1}\fi
\ifx \beditor  \undefined \def \beditor#1{#1}\fi
\ifx \bpublisher  \undefined \def \bpublisher#1{#1}\fi
\ifx \bbtitle  \undefined \def \bbtitle#1{#1}\fi
\ifx \bedition  \undefined \def \bedition#1{#1}\fi
\ifx \bseriesno  \undefined \def \bseriesno#1{#1}\fi
\ifx \blocation  \undefined \def \blocation#1{#1}\fi
\ifx \bsertitle  \undefined \def \bsertitle#1{#1}\fi
\ifx \bsnm \undefined \def \bsnm#1{#1}\fi
\ifx \bsuffix \undefined \def \bsuffix#1{#1}\fi
\ifx \bparticle \undefined \def \bparticle#1{#1}\fi
\ifx \barticle \undefined \def \barticle#1{#1}\fi
\bibcommenthead
\ifx \bconfdate \undefined \def \bconfdate #1{#1}\fi
\ifx \botherref \undefined \def \botherref #1{#1}\fi
\ifx \url \undefined \def \url#1{\textsf{#1}}\fi
\ifx \bchapter \undefined \def \bchapter#1{#1}\fi
\ifx \bbook \undefined \def \bbook#1{#1}\fi
\ifx \bcomment \undefined \def \bcomment#1{#1}\fi
\ifx \oauthor \undefined \def \oauthor#1{#1}\fi
\ifx \citeauthoryear \undefined \def \citeauthoryear#1{#1}\fi
\ifx \endbibitem  \undefined \def \endbibitem {}\fi
\ifx \bconflocation  \undefined \def \bconflocation#1{#1}\fi
\ifx \arxivurl  \undefined \def \arxivurl#1{\textsf{#1}}\fi
\csname PreBibitemsHook\endcsname

\bibitem{avriel1968golden}
\begin{barticle}
\bauthor{\bsnm{Avriel}, \binits{M.}},
\bauthor{\bsnm{Wilde}, \binits{D.J.}}:
\batitle{Golden block search for the maximum of unimodal functions}.
\bjtitle{Management Science}
\bvolume{14}(\bissue{5}),
\bfpage{307}--\blpage{319}
(\byear{1968})
\end{barticle}
\endbibitem

\bibitem{nesterov2017random}
\begin{barticle}
\bauthor{\bsnm{Nesterov}, \binits{Y.}},
\bauthor{\bsnm{Spokoiny}, \binits{V.}}:
\batitle{Random gradient-free minimization of convex functions}.
\bjtitle{Found. Comput. Math.}
\bvolume{17}(\bissue{2}),
\bfpage{527}--\blpage{566}
(\byear{2017})
\end{barticle}
\endbibitem

\bibitem{Stich_2013}
\begin{barticle}
\bauthor{\bsnm{Stich}, \binits{S.U.}},
\bauthor{\bsnm{M\"{u}ller}, \binits{C.L.}},
\bauthor{\bsnm{G\"{a}rtner}, \binits{B.}}:
\batitle{Optimization of convex functions with random pursuit}.
\bjtitle{SIAM J. Optim.}
\bvolume{23}(\bissue{2}),
\bfpage{1284}--\blpage{1309}
(\byear{2013})
\end{barticle}
\endbibitem

\bibitem{stich2014convex}
\begin{botherref}
\oauthor{\bsnm{Stich}, \binits{S.U.}}:
Convex optimization with random pursuit.
PhD thesis,
ETH Zurich
(2014)
\end{botherref}
\endbibitem

\bibitem{robbins1951stochastic}
\begin{botherref}
\oauthor{\bsnm{Robbins}, \binits{H.}},
\oauthor{\bsnm{Monro}, \binits{S.}}:
A stochastic approximation method.
Ann. Math. Stat.,
400--407
(1951)
\end{botherref}
\endbibitem

\bibitem{bfgs}
\begin{bbook}
\bauthor{\bsnm{Fletcher}, \binits{R.}}:
\bbtitle{Newton-like methods},
pp. \bfpage{44}--\blpage{74}.
\bpublisher{Wiley}, \blocation{???}
(\byear{1986})
\end{bbook}
\endbibitem

\bibitem{more2009benchmarking}
\begin{barticle}
\bauthor{\bsnm{Mor{\'{e}}}, \binits{J.J.}},
\bauthor{\bsnm{Wild}, \binits{S.M.}}:
\batitle{Benchmarking derivative-free optimization algorithms}.
\bjtitle{{SIAM} Journal on Optimization}
\bvolume{20}(\bissue{1}),
\bfpage{172}--\blpage{191}
(\byear{2009})
\end{barticle}
\endbibitem

\bibitem{nelder1965simplex}
\begin{barticle}
\bauthor{\bsnm{Nelder}, \binits{J.A.}},
\bauthor{\bsnm{Mead}, \binits{R.}}:
\batitle{A simplex method for function minimization}.
\bjtitle{Comput. J.}
\bvolume{7}(\bissue{4}),
\bfpage{308}--\blpage{313}
(\byear{1965})
\end{barticle}
\endbibitem

\bibitem{wolfe1969convergence}
\begin{barticle}
\bauthor{\bsnm{Wolfe}, \binits{P.}}:
\batitle{Convergence conditions for ascent methods}.
\bjtitle{SIAM Rev.}
\bvolume{11}(\bissue{2}),
\bfpage{226}--\blpage{235}
(\byear{1969})
\end{barticle}
\endbibitem

\bibitem{recht-wright}
\begin{bbook}
\bauthor{\bsnm{Wright}, \binits{S.J.}},
\bauthor{\bsnm{Recht}, \binits{B.}}:
\bbtitle{Optimization for Data Analysis}.
\bpublisher{Cambridge University Press},
\blocation{Cambridge}
(\byear{2022}).
\doiurl{10.1017/9781009004282}
\end{bbook}
\endbibitem

\bibitem{Bengio_2012}
\begin{botherref}
\oauthor{\bsnm{Bengio}, \binits{Y.}}:
Practical recommendations for gradient-based training of deep architectures.
Neural Networks: Tricks of the Trade,
437--478
(2012)
\end{botherref}
\endbibitem

\bibitem{Polyak87}
\begin{bbook}
\bauthor{\bsnm{Polyak}, \binits{B.T.}}:
\bbtitle{Introduction to Optimization}.
\bpublisher{Optimization Software},
\blocation{New York}
(\byear{1987})
\end{bbook}
\endbibitem

\bibitem{hazan2019revisiting}
\begin{botherref}
\oauthor{\bsnm{Hazan}, \binits{E.}},
\oauthor{\bsnm{Kakade}, \binits{S.}}:
Revisiting the {Polyak} step size
(2019)
\end{botherref}
\endbibitem

\bibitem{bb}
\begin{barticle}
\bauthor{\bsnm{Barzilai}, \binits{J.}},
\bauthor{\bsnm{Borwein}, \binits{J.J.}}:
\batitle{Two-point step size gradient methods}.
\bjtitle{IMA J. Numer. Anal.}
\bvolume{8},
\bfpage{141}--\blpage{148}
(\byear{1988})
\end{barticle}
\endbibitem

\bibitem{paquette2018stochastic}
\begin{botherref}
\oauthor{\bsnm{Paquette}, \binits{C.}},
\oauthor{\bsnm{Scheinberg}, \binits{K.}}:
A stochastic line search method with convergence rate analysis
(2018)
\end{botherref}
\endbibitem

\bibitem{berahas2019global}
\begin{botherref}
\oauthor{\bsnm{Berahas}, \binits{A.S.}},
\oauthor{\bsnm{Cao}, \binits{L.}},
\oauthor{\bsnm{Scheinberg}, \binits{K.}}:
Global convergence rate analysis of a generic line search algorithm with noise
(2019)
\end{botherref}
\endbibitem

\bibitem{orabona2016coin}
\begin{botherref}
\oauthor{\bsnm{Orabona}, \binits{F.}},
\oauthor{\bsnm{P{\'{a}}l}, \binits{D.}}:
Coin betting and parameter-free online learning
(2016)
\end{botherref}
\endbibitem

\bibitem{cutkosky2017online}
\begin{botherref}
\oauthor{\bsnm{Cutkosky}, \binits{A.}},
\oauthor{\bsnm{Boahen}, \binits{K.}}:
Online learning without prior information
(2017)
\end{botherref}
\endbibitem

\bibitem{Duchi:2011:ASM:1953048.2021068}
\begin{barticle}
\bauthor{\bsnm{Duchi}, \binits{J.}},
\bauthor{\bsnm{Hazan}, \binits{E.}},
\bauthor{\bsnm{Singer}, \binits{Y.}}:
\batitle{Adaptive subgradient methods for online learning and stochastic
  optimization}.
\bjtitle{J. Mach. Learn. Res.}
\bvolume{12},
\bfpage{2121}--\blpage{2159}
(\byear{2011})
\end{barticle}
\endbibitem

\bibitem{orabona2017training}
\begin{botherref}
\oauthor{\bsnm{Orabona}, \binits{F.}},
\oauthor{\bsnm{Tommasi}, \binits{T.}}:
Training deep networks without learning rates through coin betting
(2017)
\end{botherref}
\endbibitem

\bibitem{Bollapragada_2018}
\begin{barticle}
\bauthor{\bsnm{Bollapragada}, \binits{R.}},
\bauthor{\bsnm{Byrd}, \binits{R.}},
\bauthor{\bsnm{Nocedal}, \binits{J.}}:
\batitle{Adaptive sampling strategies for stochastic optimization}.
\bjtitle{SIAM J. Optim.}
\bvolume{28}(\bissue{4}),
\bfpage{3312}--\blpage{3343}
(\byear{2018})
\end{barticle}
\endbibitem

\bibitem{bollapragada2018progressive}
\begin{botherref}
\oauthor{\bsnm{Bollapragada}, \binits{R.}},
\oauthor{\bsnm{Mudigere}, \binits{D.}},
\oauthor{\bsnm{Nocedal}, \binits{J.}},
\oauthor{\bsnm{Shi}, \binits{H.-J.M.}},
\oauthor{\bsnm{Tang}, \binits{P.T.P.}}:
A progressive batching L-BFGS method for machine learning
(2018)
\end{botherref}
\endbibitem

\bibitem{Friedlander_2012}
\begin{barticle}
\bauthor{\bsnm{Friedlander}, \binits{M.P.}},
\bauthor{\bsnm{Schmidt}, \binits{M.}}:
\batitle{Hybrid deterministic-stochastic methods for data fitting}.
\bjtitle{SIAM J. Sci. Comput.}
\bvolume{34}(\bissue{3}),
\bfpage{1380}--\blpage{1405}
(\byear{2012})
\end{barticle}
\endbibitem

\bibitem{schmidt2017minimizing}
\begin{barticle}
\bauthor{\bsnm{Schmidt}, \binits{M.}},
\bauthor{\bsnm{Le~Roux}, \binits{N.}},
\bauthor{\bsnm{Bach}, \binits{F.}}:
\batitle{Minimizing finite sums with the stochastic average gradient}.
\bjtitle{Math. Program.}
\bvolume{162}(\bissue{1--2}),
\bfpage{83}--\blpage{112}
(\byear{2017})
\end{barticle}
\endbibitem

\bibitem{schaul2012pesky}
\begin{botherref}
\oauthor{\bsnm{Schaul}, \binits{T.}},
\oauthor{\bsnm{Zhang}, \binits{S.}},
\oauthor{\bsnm{LeCun}, \binits{Y.}}:
No more pesky learning rates
(2012)
\end{botherref}
\endbibitem

\bibitem{almeida_langlois_amaral_plakhov_1999}
\begin{bbook}
\bauthor{\bsnm{Almeida}, \binits{L.B.}},
\bauthor{\bsnm{Langlois}, \binits{T.}},
\bauthor{\bsnm{Amaral}, \binits{J.D.}},
\bauthor{\bsnm{Plakhov}, \binits{A.}}:
\bbtitle{Parameter adaptation in stochastic optimization}.
\bsertitle{Publications of the Newton Institute},
pp. \bfpage{111}--\blpage{134}.
\bpublisher{Cambridge Univ. Press}, \blocation{???}
(\byear{1999})
\end{bbook}
\endbibitem

\bibitem{vaswani2019painless}
\begin{botherref}
\oauthor{\bsnm{Vaswani}, \binits{S.}},
\oauthor{\bsnm{Mishkin}, \binits{A.}},
\oauthor{\bsnm{Laradji}, \binits{I.}},
\oauthor{\bsnm{Schmidt}, \binits{M.}},
\oauthor{\bsnm{Gidel}, \binits{G.}},
\oauthor{\bsnm{Lacoste-Julien}, \binits{S.}}:
Painless stochastic gradient: Interpolation, line-search, and convergence rates
(2019)
\end{botherref}
\endbibitem

\bibitem{DFObook}
\begin{bbook}
\bauthor{\bsnm{Conn}, \binits{A.R.}},
\bauthor{\bsnm{Scheinberg}, \binits{K.}},
\bauthor{\bsnm{Vicente}, \binits{L.N.}}:
\bbtitle{Introduction to Derivative-free Optimization}.
\bsertitle{{MPS-SIAM} Ser. Optim.},
vol. \bseriesno{8}.
\bpublisher{{SIAM}}, \blocation{???}
(\byear{2009})
\end{bbook}
\endbibitem

\bibitem{jamieson2012query}
\begin{bchapter}
\bauthor{\bsnm{Jamieson}, \binits{K.G.}},
\bauthor{\bsnm{Nowak}, \binits{R.}},
\bauthor{\bsnm{Recht}, \binits{B.}}:
\bctitle{Query complexity of derivative-free optimization}.
In: \bbtitle{Advances in Neural Information Processing Systems},
pp. \bfpage{2672}--\blpage{2680}
(\byear{2012})
\end{bchapter}
\endbibitem

\bibitem{golovin2020gradientless}
\begin{botherref}
\oauthor{\bsnm{Golovin}, \binits{D.}},
\oauthor{\bsnm{Karro}, \binits{J.}},
\oauthor{\bsnm{Kochanski}, \binits{G.}},
\oauthor{\bsnm{Lee}, \binits{C.}},
\oauthor{\bsnm{Song}, \binits{X.}},
\oauthor{\bsnm{Zhang}, \binits{Q.}}:
Gradientless descent: High-dimensional zeroth-order optimization
(2020)
\end{botherref}
\endbibitem

\bibitem{berahas2019theoretical}
\begin{botherref}
\oauthor{\bsnm{Berahas}, \binits{A.S.}},
\oauthor{\bsnm{Cao}, \binits{L.}},
\oauthor{\bsnm{Choromanski}, \binits{K.}},
\oauthor{\bsnm{Scheinberg}, \binits{K.}}:
A theoretical and empirical comparison of gradient approximations in
  derivative-free optimization
(2019)
\end{botherref}
\endbibitem

\bibitem{nocedal2006line}
\begin{botherref}
\oauthor{\bsnm{Nocedal}, \binits{J.}},
\oauthor{\bsnm{Wright}, \binits{S.J.}}:
Line search methods.
Numer. Optim.,
30--65
(2006)
\end{botherref}
\endbibitem

\bibitem{nesterov2013gradient}
\begin{barticle}
\bauthor{\bsnm{Nesterov}, \binits{Y.}}:
\batitle{Gradient methods for minimizing composite functions}.
\bjtitle{Math. Program.}
\bvolume{140}(\bissue{1}),
\bfpage{125}--\blpage{161}
(\byear{2013})
\end{barticle}
\endbibitem

\bibitem{armijo1966minimization}
\begin{barticle}
\bauthor{\bsnm{Armijo}, \binits{L.}}:
\batitle{Minimization of functions having {Lipschitz} continuous first partial
  derivatives}.
\bjtitle{Pacific J. Math.}
\bvolume{16}(\bissue{1}),
\bfpage{1}--\blpage{3}
(\byear{1966})
\end{barticle}
\endbibitem

\bibitem{brent1971algorithm}
\begin{barticle}
\bauthor{\bsnm{Brent}, \binits{R.P.}}:
\batitle{An algorithm with guaranteed convergence for finding a zero of a
  function}.
\bjtitle{Comput. J.}
\bvolume{14}(\bissue{4}),
\bfpage{422}--\blpage{425}
(\byear{1971})
\end{barticle}
\endbibitem

\bibitem{Nesterovbook}
\begin{bbook}
\bauthor{\bsnm{Nesterov}, \binits{Y.E.}}:
\bbtitle{Introductory Lectures on Convex Optimization - {A} Basic Course}.
\bsertitle{Appl. Optim.},
vol. \bseriesno{87}.
\bpublisher{Springer}, \blocation{???}
(\byear{2004})
\end{bbook}
\endbibitem

\bibitem{libsvm}
\begin{bbook}
\bauthor{\bsnm{Platt}, \binits{J.C.}}:
\bbtitle{Fast training of support vector machines using sequential minimal
  optimization},
pp. \bfpage{185}--\blpage{208}.
\bpublisher{MIT Press},
\blocation{Cambridge, MA, USA}
(\byear{1999})
\end{bbook}
\endbibitem

\bibitem{UCI}
\begin{botherref}
\oauthor{\bsnm{Dua}, \binits{D.}},
\oauthor{\bsnm{Graff}, \binits{C.}}:
{UCI} machine learning repository
(2017).
\url{http://archive.ics.uci.edu/ml}
\end{botherref}
\endbibitem

\bibitem{Caruana:2004:KRA:1046456.1046470}
\begin{barticle}
\bauthor{\bsnm{Caruana}, \binits{R.}},
\bauthor{\bsnm{Joachims}, \binits{T.}},
\bauthor{\bsnm{Backstrom}, \binits{L.}}:
\batitle{{KDD}-cup 2004: Results and analysis}.
\bjtitle{SIGKDD Explorations Newsletter}
\bvolume{6}(\bissue{2}),
\bfpage{95}--\blpage{108}
(\byear{2004})
\end{barticle}
\endbibitem

\bibitem{scipy}
\begin{barticle}
\bauthor{\bsnm{{Virtanen}}, \binits{P.}},
\bauthor{\bsnm{{Gommers}}, \binits{R.}},
\bauthor{\bsnm{{Oliphant}}, \binits{T.E.}},
\bauthor{\bsnm{{Haberland}}, \binits{M.}},
\bauthor{\bsnm{{Reddy}}, \binits{T.}},
\bauthor{\bsnm{{Cournapeau}}, \binits{D.}},
\bauthor{\bsnm{{Burovski}}, \binits{E.}},
\bauthor{\bsnm{{Peterson}}, \binits{P.}},
\bauthor{\bsnm{{Weckesser}}, \binits{W.}},
\bauthor{\bsnm{{Bright}}, \binits{J.}},
\bauthor{\bsnm{{van der Walt}}, \binits{S.J.}},
\bauthor{\bsnm{{Brett}}, \binits{M.}},
\bauthor{\bsnm{{Wilson}}, \binits{J.}},
\bauthor{\bsnm{{Jarrod Millman}}, \binits{K.}},
\bauthor{\bsnm{{Mayorov}}, \binits{N.}},
\bauthor{\bsnm{{Nelson}}, \binits{A.R.J.}},
\bauthor{\bsnm{{Jones}}, \binits{E.}},
\bauthor{\bsnm{{Kern}}, \binits{R.}},
\bauthor{\bsnm{{Larson}}, \binits{E.}},
\bauthor{\bsnm{{Carey}}, \binits{C.}},
\bauthor{\bsnm{{Polat}}, \binits{{\. I}.}},
\bauthor{\bsnm{{Feng}}, \binits{Y.}},
\bauthor{\bsnm{{Moore}}, \binits{E.W.}},
\bauthor{\bsnm{{Vand erPlas}}, \binits{J.}},
\bauthor{\bsnm{{Laxalde}}, \binits{D.}},
\bauthor{\bsnm{{Perktold}}, \binits{J.}},
\bauthor{\bsnm{{Cimrman}}, \binits{R.}},
\bauthor{\bsnm{{Henriksen}}, \binits{I.}},
\bauthor{\bsnm{{Quintero}}, \binits{E.A.}},
\bauthor{\bsnm{{Harris}}, \binits{C.R.}},
\bauthor{\bsnm{{Archibald}}, \binits{A.M.}},
\bauthor{\bsnm{{Ribeiro}}, \binits{A.H.}},
\bauthor{\bsnm{{Pedregosa}}, \binits{F.}},
\bauthor{\bsnm{{van Mulbregt}}, \binits{P.}},
\bauthor{\bsnm{{SciPy 1. 0 Contributors}}}:
\batitle{{SciPy 1.0: Fundamental algorithms for scientific computing in
  python}}.
\bjtitle{Nature Methods}
\bvolume{17},
\bfpage{261}--\blpage{272}
(\byear{2020})
\end{barticle}
\endbibitem

\end{thebibliography}


\end{document}